\documentclass[a4paper,11pt]{article}
\usepackage[latin1]{inputenc} 
\usepackage[T1]{fontenc}
\usepackage{lmodern}

\usepackage{amsthm,amsmath,amsfonts,amssymb,stmaryrd,bbm,mathrsfs}
\usepackage{mathtools}
\usepackage{enumitem}
\usepackage[colorlinks=true, linkcolor=blue, citecolor=blue,pdfstartview=FitH]{hyperref}

\usepackage[top=2.7cm, bottom=2.7cm, left=2.5cm, right=2.5cm]{geometry}

\usepackage[english]{babel}

\usepackage{caption,tikz,color}
\usetikzlibrary{shapes}
\usetikzlibrary{patterns}

\makeatletter

\@addtoreset{equation}{section}
\makeatother

\let\originalleft\left
\let\originalright\right
\renewcommand{\left}{\mathopen{}\mathclose\bgroup\originalleft}
\renewcommand{\right}{\aftergroup\egroup\originalright}

\newlength{\bibitemsep}
\newlength{\bibparskip}
\let\oldthebibliography\thebibliography
\renewcommand\thebibliography[1]{\oldthebibliography{#1}
  \setlength{\parskip}{\bibitemsep}
  \setlength{\itemsep}{\bibparskip}}

\newcommand{\N}{\mathbb{N}}
\newcommand{\Z}{\mathbb{Z}}
\newcommand{\Q}{\mathbb{Q}}
\newcommand{\R}{\mathbb{R}}
\newcommand{\C}{\mathbb{C}}
\newcommand{\T}{\mathbb{T}}
\renewcommand{\P}{\mathbb{P}}
\newcommand{\E}{\mathbb{E}}

\newcommand{\Ec}[1]{\mathbb{E} \left[#1\right]}
\newcommand{\Pp}[1]{\mathbb{P} \left(#1\right)}
\newcommand{\Ecsq}[2]{\mathbb{E} \left[#1\middle|#2\right]}
\newcommand{\Ppsq}[2]{\mathbb{P} \left(#1\middle|#2\right)}
\newcommand{\Varsq}[2]{\Var \left(#1\middle|#2\right)}
\newcommand{\Eci}[2]{\mathbb{E}_{#1} \left[#2\right]}
\newcommand{\Ppi}[2]{\mathbb{P}_{#1} \left(#2\right)}
\newcommand{\Ecii}[3]{\mathbb{E}_{#1}^{#2} \left[#3\right]}

\newcommand{\Ppsqi}[3]{\mathbb{P}_{#1} \left(#2\middle|#3\right)}
\newcommand{\Ecsqi}[3]{\mathbb{E}_{#1} \left[#2\middle|#3\right]}

\newcommand{\1}{\mathbbm{1}}

\newcommand{\cN}{\mathcal{N}}

\newcommand{\cC}{\mathcal{C}}

\newcommand{\cE}{\mathcal{E}}

\newcommand{\cL}{\mathcal{L}}

\newcommand{\cS}{\mathcal{S}}

\newcommand{\sF}{\mathscr{F}}

\newcommand{\e}{\mathrm{e}}
\newcommand{\diff}{\mathop{}\mathopen{}\mathrm{d}}
\DeclareMathOperator{\Var}{Var}

\newcommand{\abs}[1]{\left\lvert#1\right\rvert}

\newcommand{\restreinta}{\mathclose{}|\mathopen{}}

\newcommand\relphantom[1]{\mathrel{\phantom{#1}}}

\title{1-stable fluctuations in branching Brownian motion at critical temperature I: the derivative martingale}

\author{
Pascal \textsc{Maillard}\thanks{Laboratoire de Math\'ematiques d'Orsay, Univ.~Paris--Sud, CNRS, Universit\'e Paris--Saclay, 91405 Orsay, France. E-Mail: \texttt{pascal.maillard@u-psud.fr}. Partially supported by ANR Liouville (ANR-15-CE40-0013), ANR GRAAL (ANR-14-CE25-0014) and a CRM Simons Research Fellowship.}
~and Michel \textsc{Pain}\thanks{DMA, \'Ecole Normale Sup\'erieure, PSL, CNRS, 75005 Paris, France 
\& LPSM, Sorbonne Université, Sorbonne Paris Cité, CNRS, 75005 Paris, France.
Email: \texttt{michel.pain@ens.fr}.}
}

\date{June 19, 2018}

\theoremstyle{plain}  
\newtheorem{thm}{Theorem}[section]
\newtheorem{prop}[thm]{Proposition}
\newtheorem{lem}[thm]{Lemma}
\newtheorem{cor}[thm]{Corollary}

\newtheorem{conj}{Conjecture}

\theoremstyle{definition}

\theoremstyle{remark}
\newtheorem{rem}[thm]{Remark}

\begin{document}

\maketitle

\vspace{-0.5cm}

\begin{abstract}
\noindent Let $(Z_t)_{t\geq 0}$ denote the derivative martingale of branching Brownian motion, i.e.\@ the derivative with respect to the inverse temperature of the normalized partition function at critical temperature. A well-known result by
Lalley and Sellke [\textit{Ann. Probab.}, 15(3):1052--1061, 1987] says that this martingale converges almost surely to a limit $Z_\infty$, positive on the event of survival.
In this paper, our concern is the fluctuations of the derivative martingale around its limit. A corollary of our results is the following convergence, confirming and strengthening a conjecture by Mueller and Munier [\textit{Phys. Rev. E}, 90:042143, 2014]:
\[
\sqrt{t} \left( Z_\infty - Z_t + \frac{\log t}{\sqrt{2\pi t}} Z_\infty \right)
\xrightarrow[t\to\infty]{} S_{Z_\infty},
\quad \text{in law},
\]
where $S$ is a spectrally positive 1-stable L\'evy process independent of $Z_\infty$.

In a first part of the paper, a relatively short proof of (a slightly stronger form of) this convergence is given based on the functional equation satisfied by the characteristic function of $Z_\infty$ together with tail asymptotics of this random variable.
We then set up more elaborate arguments which yield a more thorough understanding of the trajectories of the particles contributing to the fluctuations. In this way, we can upgrade our convergence result to functional convergence. This approach also sets the ground for a follow-up paper, where we study the fluctuations of more general functionals including the renormalized critical additive martingale.

All proofs in this paper are given under the hypothesis $\Ec{L(\log L)^3} < \infty$, where the random variable $L$ follows the offspring distribution of the branching Brownian motion. We believe this hypothesis to be optimal.
\end{abstract}


\section{Introduction}

Branching Brownian motion (BBM) is a branching Markov process defined as follows. 
Initially, there is a single particle at the origin. 
Each particle moves according to a Brownian motion with variance $\sigma^2>0$ and drift $\rho \in \R$, during an exponentially distributed time of parameter $\lambda > 0$ and then splits into 
a random number of new particles, chosen according to a reproduction law~$\mu$.
These new particles start the same process from their place of birth, behaving independently of the others. 
The system goes on indefinitely, unless there is no particle at some time. For a detailed and formal construction, see e.g.~\cite{chauvin91,Hardy2009}.
The study of BBM dates back to \cite{moyal62-2,adkemoyal63} and has been initially motivated by the link with the F-KPP reaction-diffusion equation, established by McKean \cite{mckean75,mckean76}. 
BBM can also be seen as a Gaussian process, with covariance associated to the underlying Galton-Watson tree, and therefore is related to the generalized random energy model, introduced by Derrida and Gardner \cite{derridagardner86}, and to the mean-field spin glasses theory (see the recent book of Bovier \cite{bovier2017}).
A main focus on BBM in the last decades has been the properties of extremal particles, which are at a distance of order $1$ from the minimum of BBM, see Bramson \cite{bramson78,bramson83}, Lalley and Sellke \cite{lalleysellke87}, Aïdékon, Berestycki, Brunet and Shi \cite{abbs2013}, Arguin, Bovier and Kistler \cite{abk2013}.
A key object for understanding their behavior is the derivative martingale studied here.
Similar results have then been obtained for logarithmically correlated Gaussian fields such as the two-dimensional Gaussian Free Field: construction of the derivative martingale \cite{drsv2014-1,drsv2014-2} and study of the extremes of these fields \cite{madaule2015,BrDiZe2016,Ding2017,biskuplouidor2016,biskuplouidor2018,biskuplouidor2014arxiv}.
In this paper, we address the \emph{fluctuations} of the derivative martingale around its limit. We establish a functional convergence in law with speed of convergence $1/\sqrt{t}$ towards a randomly time-changed spectrally positive 1-stable L\'evy process. Unlike the derivative martingale, whose law depends on the model, we believe that these fluctuations are \emph{universal} and, to our knowledge, are proven here for the first time for such a model. In a follow-up paper \cite{fluctuations2}, based on the results from this article, we study the fluctuations of more general functionals including the renormalized critical additive martingale.

\subsection{Definitions and assumptions}

Let $L$ denote a random variable on $\N \coloneqq \{ 0, 1, \dots \}$ with law $\mu$. 
Our assumptions in this paper concerning the reproduction law are
\begin{equation} \label{hypothese 1}
\Ec{L} >1
\quad \text{ and } \quad
\Ec{L \log_+^3 L} < \infty,
\end{equation}
(throughout the paper we write $\log_+ x = (\log x)\vee 0$ and $\log_+^n x = (\log_+ x)^n$).
The first inequality implies that the underlying Galton-Watson tree is supercritical and the event $S$ of survival of the population has positive probability.

Let $\cN(t)$ be the set of particles\footnote{Formally, a particle is a word on the alphabet of natural numbers, i.e. an element of $\T = \bigcup_{n=0}^\infty \N^n$.} alive at time $t$ and $X_u(t)$ the position of particle $u$ at time~$t$ or of its ancestor alive at time $t$ (if it exists).
As in the branching random walk literature \cite{aidekon2013,aidekonshi2014} and as in \cite{abbs2013} for the BBM, we choose our parameters $\lambda$, $\rho$ and $\sigma$ such that, for every $t \geq 0$,
\begin{equation} \label{hypothese 2}
\E \Biggl[ \sum_{u \in \cN(t)} \e^{-X_u(t)} \Biggr] = 1
\quad \text{ and } \quad
\E \Biggl[ \sum_{u \in \cN(t)} X_u(t) \e^{-X_u(t)} \Biggr] = 0,
\end{equation}
which is equivalent to $\sigma^2 = \rho = 2\lambda \E[L-1]$ (see \cite{abbs2013}). 
Moreover, we require that, for any $t \geq 0$,
\begin{equation} \label{hypothese 3}
\E \Biggl[ \sum_{u \in \cN(t)} X_u(t)^2 \e^{-X_u(t)} \Biggr] = t,
\end{equation}
which is equivalent to $\sigma^2 = \rho = 1$ and $\lambda = 1/ (2 \E[L-1])$. One can always reduce to these parameters by a combination of translation in space and scaling in space and time. Under these assumptions, it is well-known (and an easy consequence of the convergence of $W_t$ defined below) that 
\begin{equation}
\label{eq:minimum}
 \min_{u\in\cN(t)} X_u(t) \to +\infty,\quad\text{almost surely as $t\to\infty$.}
\end{equation}

One of the main objects of study for the BBM has been the \textit{additive martingales},
introduced by McKean \cite{mckean75} and defined in our setting by
\begin{align*}
W_t(\theta) \coloneqq \sum_{u \in \cN(t)} \e^{-\theta X_u(t) - \frac{(\theta-1)^2}{2}t},
\quad t \geq 0,\ \theta\ge 0.
\end{align*}
For every $\theta$, the process $(W_t(\theta))_{t\ge0}$ is a positive martingale and therefore converges almost surely (a.s.) towards a limit $W_\infty(\theta)$. This limit is non-zero with positive probability if and only if $\theta < 1$, see \cite{neveu87} for BBM or \cite{biggins77,lyons95} for the branching random walk\footnote{One can apply results on branching random walks to BBM because a discrete-time skeleton of BBM is a branching random walk.}.
In particular, for the critical inverse temperature $\theta_c = 1$, the additive martingale
\begin{align*}
W_t \coloneqq W_t(1) = \sum_{u \in \cN(t)} \e^{-X_u(t)},
\quad t \geq 0,
\end{align*}
converges a.s.\@ to zero and one is rather interested in the so-called \textit{derivative martingale}
\begin{align*}
Z_t \coloneqq \sum_{u \in \cN(t)} X_u(t) \e^{-X_u(t)}, 
\quad t \geq 0.
\end{align*}
Indeed, it has been proved by Lalley and Sellke \cite{lalleysellke87} for binary branching and then by Yang and Ren \cite{yangren2011} under the optimal assumption $\E[L \log_+^2 L] < \infty$ that
\begin{align} \label{eq:convergence-Z_t}
Z_t \xrightarrow[t\to\infty]{} Z_\infty, \quad \text{a.s.,}
\end{align}
and $Z_\infty>0$ a.s.\@ on the event $S$ of survival.
The limit $Z_\infty$ appears in many limit theorems on branching Brownian motion. For example, Lalley and Sellke \cite{lalleysellke87}, relying on deep results by Bramson~\cite{bramson83}, proved that the 
distributional limit of the minimum of the BBM at time $t$ is a Gumbel law randomly shifted by $\log Z_\infty$:
\begin{align} \label{eq:cv-du-min-du-BBM}
\Pp{\min_{u\in\cN(t)} X_u(t) \geq \frac{3}{2} \log t + x} 
\xrightarrow[t\to\infty]{} 
\Ec{\e^{-c^* \e^x Z_\infty}},
\end{align}
for some positive constant $c^*$ (see also Aïdékon \cite{aidekon2013} for a proof under the optimal assumption $\E[L\log_+^2L]<\infty$, but for the branching random walk).
The critical additive martingale $W = (W_t)_{t\ge0}$ is also related to the derivative martingale by the following convergence
\begin{align} \label{eq:convergence-W_t}
\sqrt{t} W_t \xrightarrow[t\to\infty]{} \sqrt{\frac{2}{\pi}} Z_\infty, 
\quad \text{in probability},
\end{align}
proved for the branching random walk by A\"idékon and Shi \cite{aidekonshi2014}.
Their result applies to the BBM under the assumption $\E[L \log_+^2 L] < \infty$.
Some precise estimates have been proved recently for the tail of $Z_\infty$.
Berestycki, Berestycki and Schweinsberg \cite{bbs2013} proved in the case of binary branching that
\begin{align} \label{eq:tail-Z_infty}
\Pp{Z_\infty > x} \underset{x\to\infty}{\sim} \frac{1}{x}
\end{align}
and also that, for some constant $c_Z \in \R$, depending on the offspring distribution $\mu$,
\begin{align} \label{eq:tail-Z_infty-2}
\Ec{Z_\infty \1_{Z_\infty \leq x}} - \log x
\xrightarrow[x\to\infty]{} c_Z.
\end{align}
Moreover, Maillard \cite[Chap.\@ 2 Prop.\@ 4.1]{maillard2012thesis} proved that \eqref{eq:tail-Z_infty} holds as soon as $\E[L \log_+^2 L] < \infty$ and \eqref{eq:tail-Z_infty-2} holds if $\E[L \log_+^3 L] < \infty$. 
See also Buraczewski \cite{buraczewski2009} and Madaule \cite{madaule2016arxiv} for \eqref{eq:tail-Z_infty} in the case of the branching random walk.

As can be seen from \eqref{eq:tail-Z_infty}, 1-stable laws will play an important role and we recall now some definitions and facts. The totally asymmetric to the right $1$-stable distribution $\cS_1(\sigma,\mu)$, with parameters $\sigma > 0$ and $\mu \in \R$, has characteristic function $$
\Psi_{\sigma,\mu}(\lambda) = \exp(-\psi_{\sigma,\mu}(\lambda)),
$$
where (see e.g.~Samorodnitsky and Taqqu~\cite{samorodnitskytaqqu1994}):
\begin{align}
\label{eq:Psi_sigma_mu}
\psi_{\sigma,\mu}(\lambda) = 
\sigma \abs{\lambda} 
	\left[ 
	1 + i \frac{2}{\pi} \mathrm{sign}(\lambda) \log \abs{\lambda}
	\right]
	- i \mu \lambda
,\quad \lambda\in\R.
\end{align}
Since it is an infinitely divisible distribution, there is an associated Lévy process $(S_r)_{r\geq 0}$ called spectrally positive 1-stable process with parameters $(\sigma,\mu)$, starting at $S_0 = 0$. The characteristic functions of its one-dimensional marginals are given by
\begin{align}
 \label{eq:Psi_S_t}
 \Psi_{S_t}(\lambda) = \exp(-t \psi_{\sigma,\mu}(\lambda)) = \Psi_{t\sigma,t\mu}(\lambda),\quad t\ge0,\,\lambda\in\R.
\end{align}
In other words, $S_t$ follows the distribution $\cS_1(t\sigma,t\mu)$.

Denote by $\Psi_{Z_\infty}$ the characteristic function of $Z_\infty$, i.e. $\Psi_{Z_\infty}(\lambda) = \E[e^{i\lambda Z_\infty}]$, $\lambda\in\R$. Equations~\eqref{eq:tail-Z_infty} and \eqref{eq:tail-Z_infty-2} together with Lemma~\ref{lem:asymptotic} yield the following asymptotic for $\Psi_{Z_\infty}$: there exists a continuous function $g \colon \R \to \C$, with $g(0) = 0$, such that for every sufficiently small $\lambda$,
\begin{align} \label{eq:characteristic-function-Z_infty}
\Psi_{Z_\infty}(\lambda)
= \Psi_{\pi/2,\mu_Z}(\lambda) \e^{\lambda g(\lambda)},
\end{align}
where $\mu_Z = c_Z -\gamma$  with $c_Z$ the constant in \eqref{eq:tail-Z_infty-2} and $\gamma$ the Euler-Mascheroni constant.
In particular, the distribution of $Z_\infty$ belongs to the domain
of attraction of a totally asymmetric (to the right) $1$-stable law.

\subsection{Results}

Recall that our assumptions are \eqref{hypothese 1}, \eqref{hypothese 2} and \eqref{hypothese 3}.
Our main result is the functional convergence in law of the fluctuations of the derivative martingale, conditionally on the past. The notion of weak convergence in probability is recalled in Section~\ref{section:weak-convergence-in-probability}.
\begin{thm} \label{theorem}
Let $(S_t)_{t\ge0}$ denote a spectrally positive 1-stable Lévy process with parameters $(\sqrt{\pi/2},\mu_Z \sqrt{2/\pi})$, independent of $Z_\infty$. 
Then the conditional law of $(\sqrt{t} (Z_\infty - Z_{a t} +\frac{\log t}{\sqrt{2 \pi a t}} Z_\infty ))_{a\geq1}$ given $\sF_t$ converges weakly in probability (in the sense of finite-dimensional distributions) to the conditional law of $(S_{Z_\infty/\sqrt a})_{a\ge1}$ given $Z_\infty$. 
In other words, for every $n \geq 1$, $a_1, \dots, a_n \in [1,\infty)$ and $f \colon \R^n \to \R$ bounded and continuous, we have
\begin{align*}
& \Ecsq{ f \left( \sqrt{t} 
\left( Z_\infty - Z_{a_k t} + \frac{\log t}{\sqrt{2 \pi a_k t}} Z_\infty \right), 
1 \leq k \leq n
\right)}{\sF_t} \\
& \xrightarrow[t\to\infty]{} 
\Ecsq{f \left(S_{Z_\infty/\sqrt{a_k}}, 1 \leq k \leq n \right)}
	{Z_\infty},\quad\text{in probability.}
\end{align*}
In particular (take $n=1$ and $a_1=1$), the conditional law of $\sqrt{t} (Z_\infty - Z_t +\frac{\log t}{\sqrt{2 \pi t}} Z_\infty)$ given $\sF_t$ converges weakly in probability to the law $\cS_1(Z_\infty \sqrt{\pi/2},Z_\infty \mu_Z \sqrt{2/\pi})$ given $Z_\infty$.
\end{thm}

\begin{rem}
 One may wonder whether one actually has almost sure weak convergence in Theorem~\ref{theorem} instead of mere weak convergence in probability. This is not the case and due to the fact that the convergence in \eqref{eq:convergence-W_t} does not hold almost surely \cite{aidekonshi2014}.
\end{rem}

Removing the conditioning, we get the following corollary:
\begin{cor}
Let $(S_t)_{t\ge0}$ denote a spectrally positive 1-stable Lévy process with parameters $(\sqrt{\pi/2},\mu_Z \sqrt{2/\pi})$, independent of $Z_\infty$. 
Then, we have the following convergence in law with respect to finite-dimensional distributions:
\[
 \left(\sqrt{t} \left(Z_\infty - Z_{a t} +\frac{\log t}{\sqrt{2 \pi a t}} Z_\infty \right)\right)_{a\geq1}  \xrightarrow[t\to\infty]{\text{(law)}} \left(S_{Z_\infty/\sqrt a}\right)_{a\ge 1}.
\]
In particular, we have the following convergence in law:
\label{corollary}
 \begin{align*}
\sqrt{t} \left( Z_\infty - Z_t +\frac{\log t}{\sqrt{2 \pi t}} Z_\infty \right)
& \xrightarrow[t\to\infty]{\text{(law)}} 
\cS_1 \left( Z_\infty \sqrt{\pi/2},Z_\infty \mu_Z \sqrt{2/\pi} \right).
\end{align*}
\end{cor}

\begin{rem}
We believe the assumption $\E[L\log_+^3 L] < \infty$ to be optimal for our result. It should be compared to the assumption $\E[L\log_+^2 L] < \infty$ in previous results concerning the derivative martingale and the extremal particles. 
Our assumption is used to get the precise tail of $Z_\infty$ in \eqref{eq:tail-Z_infty-2} and also at three different places in Section \ref{section:preliminary-results}. It can apparently not be relaxed in any one of these places.
\end{rem}

Finally, we state a second result, giving an explicit control on the rate of convergence of $Z_t$ to $Z_\infty$. It will be proved with the same tools and can be of independent interest.
\begin{prop} \label{prop:speed-of-CV-Z_t}
There exists $C > 0$ such that, for any $0 < \delta \leq 1$ and $t \geq 2$, we have
\[
\Pp{\abs{Z_\infty-Z_t} \geq \delta}
\leq C \frac{(\log t)^2}{\delta \sqrt{t}}.
\]
\end{prop}

\subsection{Comments and heuristics}
\label{subsection:comments-and-heuristics}

Our motivation for studying the fluctuations of the derivative martingale $Z_t$ came from an article by Mueller and Munier \cite{muellermunier2014} in the physics literature. 
In this article, the authors mainly work with the additive martingale $W_t$. Their findings can be interpreted as the following conjecture:
\begin{align} \label{eq:one-dimensional-cv-for-W}
\sqrt{t} \left( \sqrt{\frac{2}{\pi}} Z_\infty - \sqrt{t} W_t \right)\quad\text{converges in law, as $t\to\infty$.}
\end{align}
In their appendix, they note that, for the derivative martingale $Z_t$, a corrective term of order $(\log t)/\sqrt{t}$ has to be added to get the convergence, a conjecture which they derived from numerical simulations.

Mueller and Munier give a phenomenological description of BBM in order to support their conjectures concerning the fluctuations of the front of BBM%
\footnote{We call here the \textit{front} of BBM the particles that mainly contribute to $Z_t$ (and to $W_t$): these are the particles at a position of order $\sqrt{t}$ at time $t$.
The extremal particles at time $at$ for some $a >1$ (those which are around $\frac{3}{2}\log(at) + O(1)$) mostly descend from particles in the front at time $t$.}. 
This picture is as follows: at early times of order $O(1)$, there are fluctuations due to the few number of particles, to whom the randomness of $Z_\infty$ is due. 
This random variable $Z_\infty$ determines the position of the minimum of the BBM at later times, as seen in \eqref{eq:cv-du-min-du-BBM}.
Then, after a large time of order $O(1)$, Mueller and Munier introduce the curved barrier $s \mapsto \frac{3}{2} \log s - \log Z_\infty$. Following previous works by Ebert and van Saarloos \cite{Ebert2000} on the rate of convergence of the F-KPP equation to a traveling wave, they go on to say that the density of the particles staying above this barrier can be approximated to sufficient precision by a deterministic function (with a random shift $\log Z_\infty$), whose expression is fairly intricate, involving hypergeometric functions. Mueller and Munier then argue that this density is randomly perturbed by the descendants of the particles that go below the barrier, in the spirit of previous works by Brunet, Derrida, Mueller and Munier \cite{Brunet2006}. Using the explicit form of the particle density and the stipulated law of the perturbations, Mueller and Munier are then able to derive \eqref{eq:one-dimensional-cv-for-W}. We stress that their approach, although ingenious, relies on several unjustified assumptions and approximations and uses quite intricate algebra.

Our approach is loosely inspired by Mueller and Munier \cite{muellermunier2014} but has several important differences. First, we found that instead of working with the martingale $W_t$, it is easier to work with the derivative martingale $Z_t$, which is the object of this article. Second, as in Mueller and Munier \cite{muellermunier2014} we introduce a killing barrier but which is very different from theirs. Our barrier, instead of ending at time $t$, \emph{starts at time $t$} and stays at a \emph{fixed position} $\gamma_t = \frac{1}{2} \log t + \beta_t$, for some slowly increasing function $\beta_t$ (one might think of it as a large constant $K$ and first let $t$, then $K$ go to infinity). 
The advantage of working with $Z_t$ is that the translated derivative martingale
\[
\sum_{u\in \cN(s)} (X_u(s) - \gamma_t) \e^{-X_u(t)} = Z_s - \gamma_t W_s,\quad s\ge t,
\]
is also a martingale when particles are killed when going below $\gamma_t$. With this killing, we show that the fluctuations of this martingale are of order $o(1/\sqrt t)$ and therefore negligible. Roughly speaking, this allows us to write $Z_\infty$ as the sum 
\[
 Z_\infty = Z_t - \gamma_t W_t + F_t + o(1/\sqrt t),
\]
where $F_t$ is the contribution to $Z_\infty$ from the particles going below the barrier. Estimating the number of particles hitting the barrier and using the tail asymptotics of $Z_\infty$ provided by \eqref{eq:tail-Z_infty} and \eqref{eq:tail-Z_infty-2} allows to obtain precise asymptotics on the characteristic function of $F_t$, from which one can derive the one-dimensional ($n=1$, $a_1=1$) case of Theorem \ref{theorem}. The functional limit requires slightly more work but follows along the same lines. More details of the proof are exposed in Section~\ref{section:proof-of-theorem}.

We remark that our proof shows that the fluctuations of the derivative martingale are due to the particles that come down exceptionally low: around $\frac{1}{2} \log t + O(1)$.
In fact, it is well-known known \cite{hushi2009,hu2016} that, although the minimum of the BBM is most of the time around $\frac{3}{2} \log t$, one has 
\[
\liminf_{t\to\infty} \frac{1}{\log t} \min_{u\in\cN(t)} X_u(t) = \frac{1}{2},
\quad \P\text{-a.s. on } S.
\]
These rare particles are exactly the ones leading to the limit in Theorem~\ref{theorem}.

We also remark that the one-dimensional case of Theorem~\ref{theorem} can by obtained by simpler means, namely by exploiting the decomposition
\[
 Z_\infty = \sum_{u\in \cN(t)} \e^{-X_u(t)} Z_\infty^u,
\]
where $(Z_\infty^u)_{u\in\cN(t)}$ are independent copies of $Z_\infty$, independent of $(X_u(t))_{u\in\cN(t)}$. This is done in Section~\ref{section:short-proof-of-theorem}. However, this method does not allow to obtain a functional convergence and also does not explain which particles contribute to the fluctuations.


\paragraph{Further outlook.}

In an upcoming work \cite{fluctuations2}, we consider random variables of the form
\[
 Z_t(f) = \sum_{u\in\cN(t)} X_u(t) \e^{-X_u(t)} f\left(\frac{X_u(t)}{\sqrt t}\right),
\]
for a large class of functions $f$. Special cases are the derivative martingale $Z_t$ and the renormalized additive martingale $\sqrt t W_t$, which correspond to $f\equiv 1$ and $f(x) = \frac 1 x$, respectively. We prove a limit theorem in law analogous to Theorem~\ref{theorem} for these random variables, relying on the results of the present paper. The limiting random variables are again of the form $S_{Z_\infty}$ for $(S_t)_{t\ge0}$ a 1-stable L\'evy process, but the asymmetry parameter of the process can be anything and depends on the function $f$. For example, in the case of the renormalized additive martingale, the process $(S_t)_{t\ge0}$ is a Cauchy process and the logarithmic correction term vanishes. In other words, we prove Mueller and Munier's conjecture \eqref{eq:one-dimensional-cv-for-W} and identify the limit.

In future work, we plan to study the fluctuations of the minimal position $M_t = \min_{u\in\cN(t)} X_u(t)$ in BBM. This should be related to the fluctuations of the martingale $W_t$. Specifically, we conjecture the following:
\begin{conj}
\label{conjecture}
 As $t\to\infty$, for some constant $C>0$,
\[
 M_t \stackrel{\mathrm{law}}{=} \frac 3 2 \log t - \log(CZ_\infty) - G + \frac 1 {\sqrt t} S_{Z_\infty} + o\left(\frac 1 {\sqrt t}\right),
\]
where $G$ is a standard Gumbel distributed random variable, $(S_t)_{t\ge0}$ is a Cauchy process and $Z_\infty$, $G$, $(S_t)_{t\ge0}$ are independent\footnote{The way to make this statement formal is in the language of \emph{mod-$\phi$-convergence} from \cite{Delbaen2015}.}.
\end{conj}

We note that a rich literature exists on $1/\sqrt t$ corrections for solutions to the F-KPP equation and related equations, see e.g.~ \cite{Ebert2000,BeBrHaRo2017,BeBrDe2018,Nolen2016,Graham2017} and the references therein. Conjecture~\ref{conjecture} can be seen as a probabilistic version of these results. Analogous to the deterministic equation, we believe that the term $\frac 1 {\sqrt t} S_{Z_\infty}$ appearing in Conjecture~\ref{conjecture} is \emph{universal}, in that it is (up to scaling, translation and the term $Z_\infty$) the same for all models in the so-called \emph{F-KPP universality class} \cite{Brunet2006}. We also believe that there is a direct relation between the deterministic and probabilistic versions and plan to make this relation explicit in future work. 

Finally, the methods used in this article seem to be applicable to the case of the \emph{branching random walk}, with probably a few additional technical difficulties for the proof of the functional convergence. 
However, we emphasize that both the main method as well as the one from Section~\ref{section:short-proof-of-theorem} rely on the precise tail asymptotics of $Z_\infty$ in \eqref{eq:tail-Z_infty-2}, which so far has only been proved for BBM.

\subsection{Related literature}

Fluctuations of martingales have been studied for the Galton-Watson process by Heyde \cite{heyde70, heyde71}, when the reproduction law belongs to the domain of attraction of a $\alpha$-stable law with $1 < \alpha \leq 2$.
In that case, one needs an exponential scaling to get the convergence of the fluctuations towards a mixture of $\alpha$-stable laws.
See also Heyde and Brown \cite{heydebrown71} for a functional convergence
and Kesten-Stigum \cite{kestenstigum66-2}, Athreya \cite{athreya68} and Asmussen and Keiding \cite{asmussenkeiding78} for multitype branching processes (but only in the case $\alpha = 2$).
More recently, similar results have been proved for the additive martingale $W_t(\theta)$ of a \emph{branching random walk} in the subcritical regime $\theta < 1$.
Rösler, Topchii and Vatutin \cite{rtv2002} show, in the more general setting of stable weighted branching processes, the convergence of the fluctuations of $W_t(\theta)$ when $W_1(\theta)$ belongs to the domain of attraction of a $\alpha$-stable law with $1 < \alpha \leq 2$, with also an exponential scaling and a mixture of $\alpha$-stable laws as limit. More related to our results is yet unpublished work by Iksanov, Kolesko and Meiners \cite{IkKoMe2018} where $W_1(\theta)$ is light-tailed (and $\theta$ may be complex). In particular, if $\theta\in (1/2,1)$, they show that the fluctuations of $W_t(\theta)$ are exponentially small in $t$ and converge in law after rescaling to $S_{Z_\infty}$, where $(S_t)_{t\ge0}$ is a $1/\theta$-stable L\'evy process independent of $Z_\infty$. 

Note that the appearance of stable processes subordinated by $Z_\infty$ in branching random walks has been observed previously in the study of the martingales $W_t(\theta)$ in the  \emph{supercritical} regime $\theta > 1$, see e.g. \cite{DL1983,Guivarc'h1990,Barral2013a}.

Functional convergence results have been obtained in the branching random walk setting by Iksanov and Kabluchko \cite{iksanovkabluchko2016}, when $\Var(W_1(\theta)) < \infty$ and $\theta < 1/2$ (in our setting), and by Iksanov, Kolesko and Meiners \cite{ikm2017arxiv}, when $\P(W_1(\theta) \geq x) \sim c x^{-\alpha}$ and $(\theta \alpha - 1)^2 < \alpha (\theta - 1)^2$ for some $\theta < 1 < \alpha < 2$.
See also, Hartung and Klimovsky \cite{hartungklimovsky2017arxiv} for the case of complex BBM.
To our knowledge, $1$-stable fluctuations have not been studied yet.
Note also that all aforementioned results exhibit an exponential scaling, while it is polynomial here.
Fluctuations of the partition function have also been studied for other models related to BBM.
For directed polymers in random environment in the $L^2$-phase, there is also a Gaussian limit but with a polynomial scaling \cite{cometsliu2017}. 
Other results have been obtained for many different spin glasses models: for example, the Sherrington-Kirkpatrick model \cite{alr87}, the REM and the $p$-spin model \cite{bkl2002}, the GREM \cite{bovierklimovsky2008}, the complex REM \cite{kabluchkoklimovsky2014} and the spherical Sherrington-Kirkpatrick model \cite{baiklee2016,baiklee2017}. In comparing to our setting, one might argue that the situation is a bit different in these models in that the fluctuation is the \emph{first} random term appearing in the large-$N$ expansion of the partition function, whereas in our setting it is the \emph{second} or even \emph{third} random term (the first ones being $Z_\infty$ and $\frac{\log t}{\sqrt{2\pi t}} Z_\infty$). However, the conditioning in Theorem~\ref{theorem} effectively shows that one can consider $Z_\infty$ as a constant, making the term $S_{Z_\infty}$ the first ``truly'' random term in the large-$t$ expansion of $Z_t$. Another argument is that the term $S_{Z_\infty}$ is (as we believe) the first \emph{universal} term in the expansion, the term $Z_\infty$ depending on the offspring distribution and, in general, on the model. It is thus fair to say that the results from this article are  of the same nature as the results for the statistical mechanics models cited above.

%

\subsection{Organization of the paper}

In Section \ref{section:proof-of-theorem}, we state another Theorem \ref{theorem-bis} which is easier to prove than Theorem \ref{theorem}.
The fact that one can get Theorem \ref{theorem} from Theorem \ref{theorem-bis} (in the case $n=1$ or in the multi-dimensional case) follows from Proposition \ref{prop:control-W_t}, which is stated in Section \ref{section:short-proof-of-theorem} and proved in Section \ref{section:control-W}.
In Section \ref{section:short-proof-of-theorem}, we give a short proof of the case $n=1$ of Theorem \ref{theorem-bis}, which corresponds to Proposition \ref{prop:one-dim}.


The rest of the paper is dedicated to the proof of Theorem \ref{theorem-bis} in the multi-dimensional case.
In Section \ref{section:proof-of-theorem}, its proof is divided into three propositions, which are proved in Sections~\ref{section:control-without-fluctuations} and \ref{section:killed-particles}, using preliminary results stated and proved in Section \ref{section:preliminary-results}.

In Appendix \ref{section:weak-convergence-in-probability}, some theoretical definitions and results concerning weak convergence in probability for random measures are given. 
Appendix \ref{section:technical-results} contains some calculations concerning Brownian motion and the 3-dimensional Bessel process used in the paper. The lemma in Appendix \ref{section:asymptotic_Psi} provides an asymptotic for the characteristic function of a random variable satisfying \eqref{eq:tail-Z_infty} and \eqref{eq:tail-Z_infty-2}.
Appendix \ref{section:rate-of-cv-derivative-martingale} contains the proof of Proposition \ref{prop:speed-of-CV-Z_t}.

Throughout the paper, $C$ denotes a positive constant that does not depend on the parameters and can change from line to line. 
For $f \colon \R_+ \to \R$ and $g \colon \R_+ \to \R_+^*$, we say that $f(t) = o(g(t))$ as $t \to \infty$ if $\lim_{t\to\infty} f(t)/g(t) = 0$ and that $f(t) = O(g(t))$ as $t \to \infty$ if $\limsup_{t\to\infty} \abs{f(t)}/g(t) < \infty$.
Moreover, $(B_t)_{t\geq0}$ denotes a standard Brownian motion and $(R_t)_{t\geq 0}$ a 3-dimensional Bessel process.

\subsection{Acknowledgements}
We gratefully thank the organizers of the workshop ``Phase transitions on random trees'', which took place at TU Dortmund on July 13-14, 2017. The discussions with several participants, which we will not explicitly name in order not to forget anyone, gave us valuable input, in particular regarding the proof of the one-dimensional case of Theorem~\ref{theorem} from Section~\ref{section:short-proof-of-theorem}.

PM acknowledges support of ANR Liouville (ANR-15-CE40-0013), ANR GRAAL (ANR-14-CE25-0014) and a CRM Simons Research Fellowship.

\section{One-dimensional marginals: a (fairly) short proof}
\label{section:short-proof-of-theorem}

In this section, we give a relatively short proof of Theorem \ref{theorem} in the case $n=1$, $a_1 = 1$. It will follow from the following two results:

%

\begin{prop} \label{prop:one-dim}
The conditional law of $\sqrt{t} (Z_\infty - Z_t +\frac{\log t}{2} W_t)$ given $\sF_t$ converges weakly in probability to the law $\cS_1(Z_\infty \sqrt{\pi/2},Z_\infty \mu_Z \sqrt{2/\pi})$ given $Z_\infty$.
\end{prop}

\begin{prop} \label{prop:control-W_t}
For any $\theta < 1/5$, we have
\[
\limsup_{t\to\infty}
\P \left( \abs{\sqrt{t} W_t - \sqrt{\frac{2}{\pi}} Z_\infty} 
	\geq t^{-\theta} \right)
= 0.
\]
\end{prop}

The main point in Proposition~\ref{prop:one-dim} is the replacement of the term $\frac{\log t}{\sqrt{2 \pi t}} Z_\infty$ by $\frac{\log t}{2} W_t$, making its proof a lot easier. Below, we give a one-page proof of Proposition~\ref{prop:one-dim}, relying  on a direct calculation of the characteristic function based on the branching property and the asymptotic \eqref{eq:characteristic-function-Z_infty} on $\Psi_{Z_\infty}$, the characteristic function of $Z_\infty$. 
These arguments are similar to those used for the Galton-Watson process \cite{heyde70,heyde71} or for the subcritical additive martingales of the branching random walk \cite{rtv2002}.
Proposition~\ref{prop:control-W_t} on the other hand is quite technical and its proof is delegated to Section~\ref{section:control-W}, which also relies on results from Section~\ref{section:preliminary-results}. Proposition~\ref{prop:control-W_t} actually also uses Proposition~\ref{prop:one-dim} as an ingredient in order to get a certain a priori control of the speed of convergence of $Z_t$ to $Z_\infty$, but this could be replaced by more technical calculations.

%

\begin{proof}[Proof of the case $n=1$, $a_1=1$ of Theorem~\ref{theorem}]
It follows immediately from Proposition~\ref{prop:one-dim}, Proposition~\ref{prop:control-W_t} and Remark~\ref{rem:weak-convergence-in-proba}.
\end{proof}

The following lemma will be convenient for the proof of Proposition~\ref{prop:one-dim} and later. For $\sigma>0$ and $\mu\in\R$, recall that $\Psi_{\sigma,\mu}$  denotes the characteristic function of the law $\cS_1(\sigma,\mu)$ defined in \eqref{eq:Psi_sigma_mu}. 
\begin{lem}
 \label{lem:Psi_algebra} Let $\sigma>0$, $\mu\in\R$ and $\lambda\in\R$. Then, for every $x>0$,
\[
 \Psi_{\sigma,\mu}(\lambda x) = \Psi_{x\sigma,x(\mu - \sigma (2/\pi) \log x)}(\lambda) = \exp\left(-x\psi_{\sigma,\mu}(\lambda)-i\lambda \frac 2 \pi \sigma  x\log x\right).
\]
\end{lem}
\begin{proof}
 Direct calculation. 
\end{proof}

\begin{proof}[Proof of Proposition~\ref{prop:one-dim}]
By Proposition \ref{prop:weak-convergence-in-probability}, it is enough to show 
that, for any $\lambda \in \R$,
%
\begin{align}
\label{eq:varphi_t}
\varphi_t(\lambda)
 \coloneqq \Ecsq{ \exp \left( i \lambda \sqrt{t} 
\left( Z_\infty - Z_t + \frac{\log t}{2} W_t \right)
\right)}{\sF_t} 
 \xrightarrow[t\to\infty]{\P} 
\Psi_{Z_\infty \sqrt{\pi/2}, Z_\infty \mu_Z \sqrt{2/\pi}}(\lambda).
\end{align}
We now fix some $\lambda \in \R$.
The crucial fact that we use is the following well-known decomposition:
\begin{equation}
\label{eq:Z_infty_decomposition}
Z_\infty = \sum_{u \in \cN(t)} \e^{-X_u(t)} Z_\infty^{(u)},
\end{equation}
where given $\sF_t$, the random variables $Z_\infty^{(u)}$ for $u \in \cN(t)$ are i.i.d.\@ copies of $Z_\infty$. Recalling that $\Psi_{Z_\infty}$ denotes the characteristic function of $Z_\infty$, the decomposition \eqref{eq:Z_infty_decomposition} yields,
\begin{equation}
\label{eq:Psi_Z_infty_decomposition}
 \Ecsq{ \exp \left( i \lambda
Z_\infty
\right)}{\sF_t} = \prod_{u \in \cN(t)} \Psi_{Z_\infty} \left(\lambda  \e^{-X_u(t)} \right).
\end{equation}
Furthermore, by the definitions of $Z_t$ and $W_t$, we have
\begin{equation}
 \label{eq:Z_tW_t}
\exp \left( i \lambda 
\left( - Z_t + \frac{\log t}{2} W_t \right)
\right)
= \prod_{u \in \cN(t)} \exp \left( i \lambda \e^{-X_u(t)}
	\left( \frac{\log t}{2} - X_u(t) \right)
	\right),
\end{equation}
and obviously this quantity is $\sF_t$-measurable. Applying \eqref{eq:Psi_Z_infty_decomposition} and \eqref{eq:Z_tW_t} with $\lambda\sqrt t$ instead of $\lambda$ and writing  $\xi_{u,t} = \sqrt t \e^{-X_u(t)}$, we get
\begin{align*}
\varphi_t(\lambda)
 = \prod_{u \in \cN(t)} \Psi_{Z_\infty} \left(\lambda \xi_{u,t} \right) 
\exp \left( i \lambda \xi_{u,t} \log \xi_{u,t}\right).
\end{align*}
Now recall that for sufficiently small $\lambda'$, say $\abs{\lambda'} \leq \varepsilon$ for some $\varepsilon > 0$, by \eqref{eq:characteristic-function-Z_infty}, $\Psi_{Z_\infty}(\lambda') = \Psi_{\pi/2,\mu_Z}(\lambda') \e^{\lambda' g(\lambda')}$, with $g$ a continuous function vanishing at 0. 
In order to apply this to the previous equation, define the event $\mathcal E_t = \{\max_{u \in \cN(t)} \abs{\lambda} \xi_{u,t} \le \varepsilon \}$. 
On $\mathcal E_t$, we get
\begin{align*}
\varphi_t(\lambda)
& = \prod_{u \in \cN(t)} \Psi_{\pi/2,\mu_Z}\left(\lambda \xi_{u,t} \right) \exp \left(\lambda \xi_{u,t} g(\lambda\xi_{u,t}) +  i \lambda \xi_{u,t}
	\log \xi_{u,t}
	\right).
\end{align*}
Applying Lemma~\ref{lem:Psi_algebra}, together with the equality $\sqrt{t} W_t = \sum_{u \in \cN(t)} \xi_{u,t}$, we get on $\mathcal E_t$,
\begin{align*}
\varphi_t(\lambda)
& = \prod_{u \in \cN(t)} \exp \left(-\xi_{u,t}\psi_{\pi/2,\mu_Z}(\lambda) + \lambda \xi_{u,t} g(\lambda\xi_{u,t})\right)\\
& = \exp\left(-\sqrt t W_t\psi_{\pi/2,\mu_Z}(\lambda) + Y_t^\lambda\right),\qquad \text{with }Y_t^\lambda \coloneqq \sum_{u \in \cN(t)} \lambda \xi_{u,t} g(\lambda\xi_{u,t}).
\end{align*}
Now note that by \eqref{eq:cv-du-min-du-BBM}, we have $\max_{u \in \cN(t)}\xi_{u,t} \to 0$ in probability and by \eqref{eq:convergence-W_t}, we have $\sqrt{t} W_t \to \sqrt{2/\pi} Z_\infty$ in probability, as $t\to\infty$. As a consequence, $\P(\mathcal E_t) \to 1$ as $t\to\infty$ and $Y_t^\lambda \to 0$ in probability, as $t\to\infty$. All of the above now shows that $\varphi_t(\lambda) \to \Psi_{Z_\infty \sqrt{\pi/2}, Z_\infty \mu_Z \sqrt{2/\pi}}(\lambda)$ in probability, as $t\to\infty$, which concludes the proof.
%
\end{proof}

\begin{rem}
What would go wrong if one were to try to extend this proof to the case $n=2$, say? What made the above proof possible was the marvelous decomposition \eqref{eq:Z_infty_decomposition} of $Z_\infty$. One can write a similar decomposition of $Z_s$ conditioned on $\sF_t$, for $s>t$, but it is much more complicated, with the presence of additional terms interplaying in a subtle way with the time-inhomogeneity of the equation (we encourage the reader to try it out!) This road therefore seems like a dead end.
\end{rem}

\section{Strategy of the proof of Theorem \ref{theorem}}
\label{section:proof-of-theorem}

In this section, we present the strategy for the proof of Theorem
\ref{theorem}. As in the one-dimensional case, we will first prove a
slightly different version of the result, namely
Theorem~\ref{theorem-bis} below, and will deduce Theorem~\ref{theorem}
from it and Proposition~\ref{prop:control-W_t}.
\begin{thm} \label{theorem-bis}
Let $(S_t)_{t\ge0}$ denote a spectrally positive 1-stable Lévy process with parameters $(\sqrt{\pi/2},\mu_Z \sqrt{2/\pi})$, independent of $Z_\infty$. 
Then the conditional law of $(\sqrt{t} (Z_\infty - Z_{a t} +\frac{\log t}{2}
W_{at} ))_{a\geq1}$ given $\sF_t$ converges weakly in probability (in the sense
of finite-dimensional distributions) to the conditional law of
$(S_{Z_\infty/\sqrt a})_{a\ge1}$ given $Z_\infty$.
\end{thm}

The proof of Theorem~\ref{theorem-bis} is split
into three propositions stated below. The method of proof also leads to
Proposition \ref{prop:speed-of-CV-Z_t}, proved in
Appendix~\ref{section:rate-of-cv-derivative-martingale}.
For $t > 0$, we set 
\[
\gamma_t \coloneqq \frac{1}{2} \log t + \beta_t,
\]
where $(\beta_t)_{t >0}$ is a family of positive numbers such that
\begin{align} \label{eq:assumption-beta_t}
\beta_t \xrightarrow[t\to\infty]{} \infty
\quad \text{and} \quad
\frac{\beta_t}{t^{1/4}} \xrightarrow[t\to\infty]{} 0.
\end{align}
In the heuristic description of the proof in Section~\ref{subsection:comments-and-heuristics}, $\beta_t$ was taken as a large constant $K$ that does not depend on $t$ and that tends to infinity after $t \to \infty$, but the choice of $\beta_t$ in \eqref{eq:assumption-beta_t} will turn out to be sufficient (which is an interesting fact and worth to point out).
In order to study $Z_{at}$ for some $a \geq 1$, we kill particles that come below $\gamma_t$ after time $at$. For this, we use the framework of \emph{stopping lines}, which have been defined in Chauvin \cite{chauvin91}. 
Let $\cL^{at,\gamma_t}$ denote the stopping line of the killed particles and,
for $u \in \cL^{at,\gamma_t}$, $\Delta_u^{at,\gamma_t}$ the time of the death of
$u$, formally\footnote{To be precise, in the definition of the stopping line $\cL^{at,\gamma_t}$, in order to be consistent with Chauvin \cite{chauvin91}, one also has to add the  information of the killing times. However, for notational convenience we omit throughout the article the killing times from the definition of stopping lines. It will always be the case that the killing times can be inferred from the context without ambiguity.},
\begin{align*}
\cL^{at,\gamma_t} 
& \coloneqq 
\{u \in \T : \exists s \geq at \text{ such that } 
u \in \cN(s), X_u(s) \leq \gamma_t \text{ and } \forall r \in [at,s), X_u(r) > \gamma_t \}, \\
\Delta_u^{at,\gamma_t} 
& \coloneqq 
\inf \{s \geq at : u \in \cN(s) \text{ and } X_u(s) \leq \gamma_t \}.
\end{align*}
We denote by $\sF_{\cL^{at,\gamma_t}}$ the $\sigma$-algebra associated with the stopping line $\cL^{at,\gamma_t}$ \cite{chauvin91}.

For the remaining particles, we consider the following random variables: for $s \geq at$, we set
\begin{align*}
\widetilde{Z}_s^{at,\gamma_t} 
\coloneqq 
\sum_{u \in \cN(s)} (X_u(s)-\gamma_t) \e^{-X_u(s)} 
\1_{\forall r \in [at,s], X_u(r) > \gamma_t}.
\end{align*}
Then $(\widetilde{Z}_s^{at,\gamma_t})_{s \geq at}$ is a non-negative martingale (see \cite{kyprianou2004}) and, therefore, it has an almost sure limit $\widetilde{Z}_\infty^{at,\gamma_t}$.
Furthermore, we have, using that $W_s \to 0$ $\P$-a.s.\@ as $s \to \infty$,
\begin{align}
Z_\infty
& = \lim_{s \to \infty} 
\sum_{v \in \cN(s)} (X_v(s)-\gamma_t) \e^{-X_v(s)}
= \widetilde{Z}_\infty^{at,\gamma_t}
+ \sum_{u \in \cL^{at,\gamma_t}} \e^{-X_u(\Delta_u^{at,\gamma_t})} Z_\infty^{(u,at,\gamma_t)},
\label{eq:decompo-Z_infty-1}
\end{align}
where we set, for $u \in \cL^{at,\gamma_t}$,
\begin{align*}
Z_\infty^{(u,at,\gamma_t)} 
& \coloneqq \lim_{s \to \infty} 
\sum_{v \in \cN(s) : u \leq v}
(X_v(s)-\gamma_t) \e^{-(X_v(s) - X_u(\Delta_u^{at,\gamma_t}))} \\
& = \lim_{s \to \infty} 
\sum_{v \in \cN(s) : u \leq v}
(X_v(s) - X_u(\Delta_u^{at,\gamma_t})) \e^{-(X_v(s) - X_u(\Delta_u^{at,\gamma_t}))},
\end{align*}
saying that $u \leq v$ if $u$ is an ancestor of particle $v$ and using again that $W_s \to 0$ $\P$-a.\@s.
Thus, by the branching property, conditionally on $\sF_{\cL^{at,\gamma_t}}$, the $Z_\infty^{(u,at,\gamma_t)}$ for $u\in\cL^{at,\gamma_t}$ are independent and have the same law as $Z_\infty$.

Looking at \eqref{eq:decompo-Z_infty-1}, our first step will be to study the limit $\widetilde{Z}_\infty^{at,\gamma_t}$, which is well concentrated around its conditional expectation given $\sF_{at}$ thanks to the killing barrier. 
Moreover, this conditional expectation is $\widetilde{Z}_{at}^{at,\gamma_t}$, which is very close to $Z_{at} - \gamma_t W_{at}$, since the particles below $\gamma_t$ at time $at$ have a negligible weight under the assumption \eqref{eq:assumption-beta_t}.
Therefore, it leads to the following result that will be shown in Section \ref{section:control-without-fluctuations}.
\begin{prop} \label{prop:control-Ztilde-infty}
For any $a \geq 1$, we have the following convergence in probability
\begin{align*}
\sqrt{t} \abs{\widetilde{Z}_\infty^{at,\gamma_t} - \left(Z_{at} - \gamma_t W_{at}\right)}
\xrightarrow[t\to\infty]{} 0.
\end{align*}
\end{prop}
Now, we want to deal with the second term on the right-hand side of \eqref{eq:decompo-Z_infty-1}.
It leads to the limiting process of Theorem \ref{theorem}, so we need to catch the dependence in $a$.
Heuristically, we want to see $\cL^{at,\gamma_t}$ as a decreasing set in $a$, with $\cL^{at,\gamma_t} \simeq \{ u \in \cL^{t,\gamma_t} : \Delta_u^{t,\gamma_t} > at\}$. 
This is not exactly the case and, therefore, we need to split $\cL^{at,\gamma_t}$ into two disjoint parts by setting
\begin{align*}
\cL^{at,\gamma_t}_\mathrm{good}
& \coloneqq 
\{u \in \cL^{at,\gamma_t} : u \in \cL^{t,\gamma_t} 
\text{ and } \Delta_u^{t,\gamma_t} = \Delta_u^{at,\gamma_t} > at \}, \\
\cL^{at,\gamma_t}_\text{bad}
& \coloneqq \cL^{at,\gamma_t} \setminus \cL^{at,\gamma_t}_\mathrm{good}.
\end{align*}
The ``good'' particles are those which contribute to the limiting process.
They satisfy $\cL^{at,\gamma_t}_\mathrm{good} = \{ u \in \cL^{t,\gamma_t}_\mathrm{good} : \Delta_u^{t,\gamma_t} > at\}$. 
Note that we have the equivalent definitions
\begin{align*}
\cL^{at,\gamma_t}_\mathrm{good}
&=
\{u \in \cL^{at,\gamma_t} : \min_{r\in[t,at]} X_u(r) > \gamma_t \}, \\
\cL^{at,\gamma_t}_\text{bad}
&= \{u \in \cL^{at,\gamma_t} : \min_{r\in[t,at]} X_u(r) \le \gamma_t \}.
\end{align*}
See Figure \ref{figure:stopping-lines} for some examples.
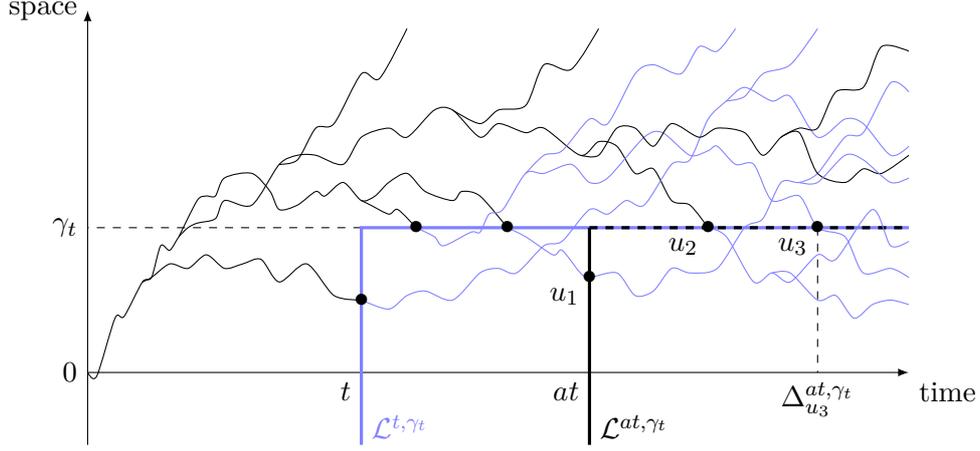
\begin{figure}[ht]
\centering
\begin{tikzpicture}[scale=1.2]
\draw[->,>=latex] (0,0) -- (9,0) node[below right]{time};
\draw[->,>=latex] (0,-0.8) -- (0,4) node[left]{space};
\draw (0,0) node[left]{$0$};
\draw[dashed] (3,1.6) -- (0,1.6) node[left]{$\gamma_t$};
\draw (3,0) node[below left]{$t$};
\draw (5.5,0) node[below left]{$at$};
\draw[dashed] (8,1.6) -- (8,0) node[below]{$\Delta_{u_3}^{at,\gamma_t}$};
\draw[blue!50,very thick] (3,-0.8) -- (3,1.6) -- (9,1.6);
\draw[very thick] (5.5,-0.8) -- (5.5,1.6);
\draw[very thick,dashed] (5.5,1.6) -- (9,1.6);
\draw[blue!50] (3,-0.6) node[right]{$\cL^{t,\gamma_t}$};
\draw (5.5,-0.6) node[right]{$\cL^{at,\gamma_t}$};
\draw[smooth] plot coordinates 
{(0,0) (0.1,-0.04) (0.3,0.6) (0.4,0.62) (0.6,1) 
(0.7,1.05) (0.8,1.4) (0.9,1.35) (1,1.5) 
(1.1,1.6) (1.2,1.65) (1.4,1.7) (1.5,1.9) (1.8,1.85) (2.1,2.3) 
(2.3,2.4) (2.5,2.4) (2.8,2.3) (3,2.6) (3.3,2.7) (3.5,2.6) (3.8,2.9) (4,2.9) 
(4.3,2.6) (4.5,2.75) (4.8,2.65) (5,2.7) (5.4,2.4) 
(5.6,2.55) (5.8,2.5) (6,2.2) (6.2,2.25) (6.4,1.9) (6.6,1.85) (6.8,1.6)}; 
\draw[smooth,blue!50] plot coordinates 
{(6.8,1.6) (7,1.5) (7.3,1) (7.5,1.1) 
(7.7,1) (8,1.3) (8.2,1.25) (8.6,1.8) (8.8,1.4) (9,1.45)
};
\draw[smooth] plot coordinates 
{(0.6,1) (0.7,1.05) (0.9,1.2) (1.1,1.15) (1.3,1.3)
(1.5,1.15) (1.7,1.2) (1.9,1.25) (2.1,1) (2.4,1.2) (2.7,0.85) (3,0.8)}; 
\draw[smooth,blue!50] plot coordinates 
{(3,0.8) (3.3,0.7) (3.5,0.85) (3.7,0.9) (4,0.8) (4.3,1.1) (4.5,1.15) (4.7,1.1) (5,1.4) (5.2,1.45) (5.5,1.8) (5.7,1.75) (5.9,1.9) (6.1,1.85) (6.4,2.3) (6.6,2.35) (6.8,2.8) (7,3) 
(7.2,3.15) (7.5,3) (7.7,3.05) (8,2.6) (8.3,2.8) (8.5,2.75) (8.8,2.4) (9,2.5)
};
\draw[smooth] plot coordinates 
{(1,1.5) (1.2, 1.9) (1.3,1.95) (1.4,1.95) (1.5,2.15) (1.8,2.2) (2,2) (2.1,1.8) (2.3,1.85)
(2.5,2) (2.6,1.95) (2.8,2.1) (3,1.9) 
(3.15,1.95) (3.3,2.1) (3.5,2.2) (3.7,2) (4,1.9) (4.1,2) (4.3,1.95) (4.6,1.6)}; 
\draw[smooth,blue!50] plot coordinates 
{(4.6,1.6) (4.8,1.5) (5.1,1.3) (5.2,1.35) (5.4,1.1) (5.5,1.05) 
(5.8,1.1) (6,0.85) (6.2,0.9) (6.4,1.15) (6.6,1.1) (6.9,1.15) (7.2,1.6) (7.4,1.7) (7.6,2.1) (7.8,2.1) 
(8,2.2) (8.3,2.6) (8.5,2.65) (8.8,3.2) (9,3.1)
};
\draw[smooth] plot coordinates 
{(2.1,2.3) (2.3,2.35) (2.5,2.7) (2.7,2.75) (3,3.3) (3.2,3.25) (3.5,3.8) 
};
\draw[smooth] plot coordinates 
{(3,1.9) (3.1,1.85) (3.3,1.75) (3.4,1.8) (3.6,1.6)}; 
\draw[smooth,blue!50] plot coordinates 
{(3.6,1.6) (3.8,1.5) (3.9,1.4) (4.2,1.55) (4.3,1.55) (4.4,1.75) (4.55,1.8) (4.8,2.2) 
(5,2.2) (5.3,2.3) (5.5,2.1) (5.7,2.15) (6.1,2.6) (6.3,2.65) (6.6,2.4) (6.8,2.45) (6.9,2.5) (7.1,2.3) (7.3,2.3) (7.5,1.9) (7.8,1.95) (8,1.6) 
(8.3,1.5) (8.5,1.1) (8.8,1.15) (9,1.05)};
\draw[smooth] plot coordinates 
{(4,2.9) (4.3,2.75) (4.5,2.9) (4.8,2.95) (5,3.4) (5.3,3.35) (5.6,3.8) 
};
\draw[smooth,blue!50] plot coordinates 
{(4.8,2.2) (5,2.4) (5.2,2.45) (5.5,2.8) (5.7,2.75) (6,3.3) (6.2,3.4) (6.5,3.3) (6.7,3.6) (6.9,3.55) (7.1,3.8) 
};
\draw[smooth] plot coordinates 
{(5.4,2.4) (5.6,2.4) (5.8,2.7) (6,2.7) (6.2,2.5) (6.4,2.75) (6.7,2.6) (6.9,2.7) (7.2,2.65) (7.4,2.4) (7.6,2.6) 
(7.8,2.7) (8,2.7) (8.2,3.1) (8.5,3.15) (8.7,3.6) (9,3.55)
};
\draw[smooth,blue!50] plot coordinates 
{(7,3) (7.2,3.05) (7.4,3.3) (7.7,3.4) (8,3.8) 
};
\draw[smooth,blue!50] plot coordinates 
{(7.5,1.1) (7.7,1.1) (8,0.8) (8.2,0.85) (8.4,0.6) (8.6,0.75) (8.8,0.8) (9,0.75)
};
\draw[smooth] plot coordinates 
{(7.6,2.6) (7.8,2.65) (8,2.2) (8.3,2.1) (8.5,2.25) (8.7,2.2) (9,2.4)
};
\draw[smooth,blue!50] plot coordinates 
{(7.8,2.1) (8,2.1) (8.3,2.4) (8.5,2.35) (8.8,2.1) (9,2.15)
};
\draw (3,0.8) node{$\bullet$};
\draw (3.6,1.6) node{$\bullet$};
\draw (4.6,1.6) node{$\bullet$};
\draw (5.5,1.05) node{$\bullet$};
\draw (5.5,1.05) node[below left]{$u_1$};
\draw (6.8,1.6) node{$\bullet$};
\draw (6.8,1.6) node[below left]{$u_2$};
\draw (8,1.6) node{$\bullet$};
\draw (8,1.6) node[below left]{$u_3$};
\end{tikzpicture}
\caption{Representation of a BBM with binary branching. 
The killing barrier that defines $\cL^{t,\gamma_t}$ is the thick blue line. 
The one for $\cL^{at,\gamma_t}$ is the thick black line.
At each time $s \geq t$, the blue particles do not contribute to $\widetilde{Z}^{t,\gamma_t}_s$.
We have 
$u_1,u_3 \in \cL_\text{bad}^{at,\gamma_t}$ and
$u_2 \in \cL_\mathrm{good}^{at,\gamma_t}$.}
\label{figure:stopping-lines}
\end{figure}

Now, we can rewrite \eqref{eq:decompo-Z_infty-1} as
\begin{align}
Z_\infty
& = \widetilde{Z}_\infty^{at,\gamma_t}
+ F^{at,\gamma_t}_\mathrm{good} + F^{at,\gamma_t}_\text{bad},
\label{eq:decompo-Z_infty-2}
\end{align}
where we set
\begin{align*}
F^{at,\gamma_t}_\mathrm{good/bad}
& \coloneqq 
\sum_{u \in \cL^{at,\gamma_t}_\mathrm{good/bad}} \e^{-\gamma_t} Z_\infty^{(u,at,\gamma_t)}.
\end{align*}
We can now state two propositions, which will be proved later in the paper.
They basically state that the ``good'' particles lead to the limiting process (after an appropriate compensation) and the ``bad'' particles have a negligible contribution.
\begin{prop} \label{prop:control-fluctuations}
Let $(S_t)_{t\ge0}$ be the L\'evy process defined in Theorem~\ref{theorem}.
For any $n \geq 1$, $a \in [1,\infty)^n$ and $\lambda \in \R^n$, we have the following convergence in probability
\begin{align*}
\Ecsq{\exp \left( i \sum_{k=1}^n \lambda_k \sqrt{t} \left( 
	F^{a_k t,\gamma_t}_\mathrm{good} - \beta_t W_{a_k t}
	\right) \right)}{\sF_t} 
 \xrightarrow[t\to\infty]{} 
\Ecsq{\exp \left( i \sum_{k=1}^n \lambda_k S_{Z\infty/\sqrt{a_k}}\right)}
	{Z_\infty}.
\end{align*}
\end{prop}
\begin{prop} \label{prop:control-F^t_bad}
For any $a \geq 1$, we have the following convergence in probability
\begin{align*}
\sqrt{t} F^{at,\gamma_t}_\mathrm{bad} \xrightarrow[t\to\infty]{} 0.
\end{align*}
\end{prop}
\begin{proof}[Proof of Theorem \ref{theorem} and Theorem \ref{theorem-bis}] 
By decomposition \eqref{eq:decompo-Z_infty-2} and the fact that $\gamma_t = \frac 1 2 \log t + \beta_t$ by definition, we have for every $a\ge1$,
\[
Z_\infty - Z_{a t} +\frac{\log t}{2} W_{at} 
= \widetilde{Z}_\infty^{at,\gamma_t} - (Z_{a t} - \gamma_t W_{a t}) + F^{at,\gamma_t}_\mathrm{good} - \beta_t W_{a t} + F^{at,\gamma_t}_\text{bad}.
\]
Propositions \ref{prop:control-Ztilde-infty} and \ref{prop:control-F^t_bad} now give for every $a\ge 1$,
\[
 \sqrt t\left(Z_\infty - Z_{a t} +\frac{\log t}{2} W_{at} \right) - \sqrt t \left(F^{at,\gamma_t}_\mathrm{good} - \beta_t W_{a t}\right) \xrightarrow[t\to\infty]{\P} 0. 
\]
Then, combining Propositions \ref{prop:control-fluctuations} and \ref{prop:weak-convergence-in-probability} with Remark \ref{rem:weak-convergence-in-proba}, this proves Theorem \ref{theorem-bis}.
We recall that Theorem \ref{theorem} follows then from Proposition \ref{prop:control-W_t}. 
\end{proof}

\section{Preliminary results on BBM with a barrier}
\label{section:preliminary-results}


\subsection{Many-to-one formula and change of probabilities}
\label{subsection:change-of-probabilities}

It will be handy in the sequel to allow the BBM to start at an arbitrary point $x\in\R$, in which case we write $\P_x$ and $\E_x$ instead of $\P$ and $\E$. The martingale property of the processes $(W_t)_{t\ge0}$ and $(Z_t)_{t\ge0}$ defined in the introduction then implies that for every $x\in\R$ and $t\ge0$,
\begin{align}
\label{eq:W_tZ_t}
\E_x[W_t] = \e^{-x},\quad \E_x[Z_t] = x\e^{-x}.
\end{align}
A generalization of this fact is provided by the following \textit{many-to-one formula}, which is an essential tool in the study of BBM. 
Let  $\cC([0,t])$ denote the space of continuous functions from $[0,t]$ to $\R$.
For any $x \in \R$, $t \geq 0$ and any measurable function $F \colon \cC([0,t]) \to \R_+$, one can compute the following expectation:
\begin{align*}
\Eci{x}{\sum_{u\in\cN(t)} \e^{-X_u(t)} F(X_u(s), s\in[0,t])}
& = \e^{-x} \Eci{x}{F(B_s, s\in[0,t])},
\end{align*}
where $(B_s)_{s \geq 0}$ denotes a standard Brownian motion, starting from $x$ under $\P_x$.
Note that this formula follows from forthcoming Proposition \ref{prop:change-of-measure-Q}.

Throughout the paper, we will use two changes of probabilities and the associated spinal decompositions. These methods have a rich history, starting with Kahane and Peyrière \cite{kahanepeyriere76}, and with its modern formulation due to Lyons, Pemantle and Peres' \cite{lpp95} work on Galton--Watson processes. We will follow the treatment by Hardy and Harris \cite{Hardy2009} for branching Markov processes. That paper supposes that the offspring distribution of the process is supported on $\{1,2,\ldots\}$, but this can be generalized, see Liu, Ren and Song \cite{Liu2011}.
See also Shi \cite{shi2015} for a survey of applications of these techniques to the branching random walk.

Let $\sF_t$ denote the $\sigma$-algebra containing all the information until time $t$, \[
\sF_t \coloneqq 
\sigma \left( \cN(s), 0 \leq s \leq t \right)
\vee
\sigma \left( X_u(s), 0 \leq s \leq t, u \in \cN(s) \right),
\] 
and $\sF_\infty \coloneqq \sigma ( \bigvee_{t\geq 0} \sF_t )$.
The first change of probability is done with respect to the critical additive martingale $(W_s)_{s\geq0}$ and has been introduced by Chauvin and Rouault \cite{chauvinrouault88}. 
Since it is a non-negative martingale, we can define, for $x >0$, a new probability measure $\Q_x$ on $\sF_\infty$ by
\begin{align*}
\Q_x \restreinta_{\sF_s}
\coloneqq
\e^x W_s \bullet \P_x \restreinta_{\sF_s}, 
\quad \forall s \geq 0.
\end{align*}
Following Hardy and Harris \cite{Hardy2009}, we rather view $\Q_x$ as the projection to $\sF_\infty$ of a probability measure (denoted by $\Q_x$ as well for simplicity) on an enlarged probability space, carrying a so-called \emph{spine}, i.e.~a marked ray in the genealogical tree. The particle on the spine by time $s$ will be denoted by $w_s$. Recall that $\lambda$ denotes the branching rate and $L$ a random variable distributed according to the offspring distribution in the original BBM. We then have the following description for the BBM with spine under $\Q_x$. 
\begin{itemize}
 \item The system starts with one particle $w_0$ at position $x$.
 \item This particle moves like a standard Brownian motion (without drift!) during a time distributed according to the exponential law of parameter $\mu_1\lambda$, where $\mu_1 := \E[L]$.
 \item Then, it gives birth to a random number $\widehat{L}$ of particles distributed according to the \textit{size-biased} reproduction law, i.e.~$\forall k \in \N, \P(\widehat{L}=k) = k \P(L=k)/\mu_1$.
 \item Amongst the children, one is uniformly chosen to be on the spine and continues in the same way. 
 \item Others spawn usual BBMs (according to the law $\P$, but started from the position of their parent).
\end{itemize}

The following proposition generalizes the many-to-one formula stated above and follows from the results in \cite{Hardy2009}. Recall that a ``particle'' is formally a word on the alphabet of natural numbers, i.e. an element of the countable set $\T = \bigcup_{n=0}^\infty \N^n$.
\begin{prop} \label{prop:change-of-measure-Q}
Let $x\in\R$.
Let $s \geq 0$ and let $(H_s(u))_{u\in \T}$ be a family of uniformly bounded $\sF_s$-measurable random variables. Then, we have 
\begin{align*}
\Eci{x}{\sum_{u\in\cN(s)} \e^{-X_u(s)} H_s(u)} = \e^{-x}\Eci{\Q_x}{H_s(w_s)}.
\end{align*}
\end{prop}

The second change of probability we will use has been introduced by Kyprianou \cite{kyprianou2004} and is done with respect to a modification of the derivative martingale, obtained by killing particles coming below 0: 
it is called the \textit{truncated derivative martingale} and is defined by
\begin{align*}
\widetilde{Z}_s
\coloneqq 
\sum_{u \in \cN(s)} X_u(s) \e^{-X_u(s)} \1_{\forall r \in [0,s], X_u(r) > 0},
\quad s \geq 0.
\end{align*}
Then, for $x\ge 0$, $(\widetilde{Z}_s)_{s\geq0}$ is a non-negative martingale under $\P_x$ and, therefore, we define a new probability measure $\widetilde{\Q}_x$ on $\sF_\infty$ by
\begin{align*}
\widetilde{\Q}_x \restreinta_{\sF_s}
\coloneqq
\frac{\e^x}{x} \widetilde{Z}_s \bullet \P_x \restreinta_{\sF_s}, 
\quad \forall s \geq 0.
\end{align*}
We have an analogous description for the BBM with spine $(w_s)_{s\ge0}$ under $\widetilde{\Q}_x$:
\begin{itemize}
 \item The system starts with one particle $w_0$ at position $x$.
 \item This particle moves like a 3-dimensional Bessel process during a time distributed according to the exponential law of parameter $\mu_1\lambda$, where $\mu_1 := \E[L]$.
 \item Then, it gives birth to a random number of particles distributed as $\widehat{L}$ defined above.
 \item Amongst the children, one is uniformly chosen to be on the spine and continues in the same way. 
 \item Others spawn usual BBMs (according to the law $\P$, but started from the position of their parent).
\end{itemize}

Moreover, one has the following generalization of the many-to-one lemma:
\begin{prop} \label{prop:change-of-measure-Qtilde}
Let $x \ge 0$.
Let $s \geq 0$ and let $(H_s(u))_{u\in \T}$ be a family of uniformly bounded $\sF_s$-measurable random variables. Then, we have
\begin{align*}
\Eci{x}{\sum_{u\in\cN(s)} X_u(s)\e^{-X_u(s)} \1_{\forall r \in [0,s], X_u(r) > 0} H_s(u)} = x\e^{-x}\Eci{\widetilde{\Q}_x}{H_s(w_s)}.
\end{align*}
\end{prop}

\subsection{Truncated derivative martingale: moments}
\label{subsection:truncated-derivative-martingale}

The second moment of the truncated derivative martingale $(\widetilde{Z}_s)_{s \geq 0}$ is infinite if $\E[L^2] = \infty$. We therefore introduce a variant of $(\widetilde{Z}_s)_{s \geq 0}$, which is close to it in $L^1$-distance and whose second moment can be controlled. This method, which uses the probability $\widetilde{\Q}_x$ defined in the previous section, is due to Aïdékon \cite{aidekon2013} and sometimes called the \emph{peeling lemma}, see Shi~\cite{shi2015} for more applications. We also incorporate simplifications of this method from Pain~\cite{Pain2017}.

Throughout, let $x\ge0$. For $u \in \T$, such that $u\in\cN(s)$ for some $s\ge0$, let $O_u$ denote the number of children of particle $u$ and $d_u$ the death time of $u$ (and leave undefined otherwise). Write $v < u$ if $v$ is a strict ancestor of $u$ and $v\le u$ if $v<u$ or $v = u$.
We set, for $s,\kappa \geq 0$, 
\begin{align*}
B_{s,\kappa} 
&\coloneqq
\left\{
u \in \cN(s) :
\forall v < u, O_v \leq \kappa \e^{X_v(d_v)/2}
\right\}, \\
\widetilde{Z}_{s,\kappa} 
& \coloneqq 
\sum_{u \in \cN(s)} X_u(s) \e^{-X_u(s)} 
\1_{\forall r \in [0,s], X_u(r) > 0} \1_{u \in B_{s,\kappa}}.
\end{align*}
Note that formally, $\widetilde{Z}_{s,\infty} = \widetilde{Z}_s$. The next lemma bounds the $L^1$-distance between $\widetilde{Z}_{s,\kappa}$ and $\widetilde{Z}_s$ for large $\kappa$.

\begin{lem} \label{lem:not-in-B}
There exists a decreasing function $h \colon \R_+ \to \R_+$ such that $\lim_{a \to \infty} h(a) = 0$ and for every $x,\kappa >0$ and $s \geq 0$,
\[
\Eci{x}{\left|\widetilde{Z}_s - \widetilde{Z}_{s,\kappa}\right|} \leq h(\kappa) \e^{-x}.
\]
\end{lem}
\begin{proof}
First note that $\widetilde{Z}_s \ge  \widetilde{Z}_{s,\kappa}$, so we can ignore the absolute value in the expectation.
Using Proposition \ref{prop:change-of-measure-Qtilde}, we get
\[
\Eci{x}{\widetilde{Z}_s - \widetilde{Z}_{s,\kappa}}
= \E_x \Biggl[
	\sum_{u \in \cN(s)} X_u(s) \e^{-X_u(s)} 
	\1_{\forall r \in [0,s], X_u(r) > 0} \1_{u \notin B_{s,\kappa}}
	\Biggr]
= x \e^{-x} \widetilde{\Q}_x \left( w_s \notin B_{s,\kappa} \right).
\]
Under $\P_x$, let $(R_r)_{r \geq 0}$ denote a 3-dimensional Bessel process and let $\widehat{L}$ be a random variable independent of $(R_r)_{r \geq 0}$ following the size-biased reproduction law.
Then, using the spinal decomposition description and the formula for
expectations of additive functionals of Poisson point processes (applied to the
branching times $r$ of the spine where the number of children exceeds
$\e^{X_{w_r}(r)/2}$), we have
\begin{align*}
\widetilde{\Q}_x \left( w_s \notin B_{s,\kappa} \right)
& \leq \Eci{x}{\int_0^s 
	\Ppsqi{x}{\widehat{L} > \kappa \e^{R_r/2}}{(R_r)_{r \geq 0}}
	\mu_1 \lambda\diff r} \\
& = \mu_1 \lambda \Eci{x}{ \Ecsqi{x}{\int_0^s \1_{R_r < 2 \log_+ (\widehat{L}/\kappa)}\diff r} {\widehat{L}} } \\
& \leq \mu_1 \lambda 
\Ec{\frac{1}{x} \left(2 \log_+ \left(\widehat{L}/\kappa \right) \right)^3},\quad\text{by \eqref{eq:time-bessel-under-a}.}
\end{align*}
The previous inequalities and the definition of $\widehat L$ yield,
\[
 \Eci{x}{\widetilde{Z}_s - \widetilde{Z}_{s,\kappa}} \le h(\kappa) \e^{-x},\quad h(\kappa) := 8 \mu_1 \lambda \E[ L (\log_+(L/\kappa))^3].
\]
The function $h(\kappa)$ is decreasing in $\kappa$ and vanishes as $\kappa \to \infty$, by dominated convergence and Assumption \eqref{hypothese 1}. This proves the result.
\end{proof}

\begin{lem} \label{lem:second-moment}
There exists $C >0$ such that for every $x\geq0$, $\kappa \ge 1$ and $s \geq 0$,
\begin{align*}
\Eci{x}{\widetilde{Z}_{s,\kappa}^2}
\leq C \kappa \e^{-x}.
\end{align*}
\end{lem}

\begin{proof}
Using Proposition \ref{prop:change-of-measure-Qtilde}, we have 
\begin{align*}
\Eci{x}{\widetilde{Z}_{s,\kappa}^2}
&= \Eci{x}{\sum_{u \in \cN(s)} X_u(s) \e^{-X_u(s)} 
	\1_{\forall r \in [0,s], X_u(r) > 0} \1_{u \notin B_{s,\kappa}} \widetilde{Z}_{s,\kappa} } \\
&= x \e^{-x} \Eci{\widetilde{\Q}_x}{\widetilde{Z}_{s,\kappa} \1_{w_s \in B_{s,\kappa}}}
\leq x \e^{-x} \Eci{\widetilde{\Q}_x}{\widetilde{Z}_s \1_{w_s \in B_{s,\kappa}}}.
\end{align*}
The $\widetilde{Z}_s$ appearing in this last expectation can be decomposed as a sum indexed by the (non-spine) children of the spine particles, each term in the sum corresponding to the contribution of the descendants of this child, plus the contribution from the spine particle at time $s$.
With $\Pi_s$ denoting the set of branching times of the spine before $s$, we then get
\begin{align*}
\Eci{\widetilde{\Q}_x}{\widetilde{Z}_s \1_{w_s \in B_{s,\kappa}}}
& = \Eci{\widetilde{\Q}_x}{\1_{w_s \in B_{s,\kappa}} 
	\sum_{r \in \Pi_s}
	\left( O_{w_r} - 1 \right)
	\Eci{X_{w_r}(r)}{\widetilde{Z}_{s-r}} + X_{w_s}(s) \e^{-X_{w_s}(s)} } \\
& \leq \Eci{x}{\int_0^s \kappa \e^{R_r/2}
	\Eci{R_r}{\widetilde{Z}_{s-r}} \mu_1 \lambda \diff r + R_s \e^{-R_s}}\\
& = \mu_1 \lambda\kappa \Eci{x}{\int_0^s R_r \e^{-R_r/2} \diff r}  + \Eci{x}{R_s \e^{-R_s}},
\end{align*}
using that $\E_{R_r}[\widetilde{Z}_{s-r}] = R_r \e^{-R_r}$. The first
expectation is bounded by $\frac{32}{x}$ by \eqref{eq:integral-bessel}. The
second expectation is bounded by $\frac {C'} x$, with $C' = \max_{y\ge0}
y^2\e^{-y}$, since the density of $R_s$ under $\P_x$ is bounded by $\frac y x
g_s(x,y)$, with $g_s(x,y)$ the Gaussian kernel of variance $s$, see
\eqref{eq:density_R}.
This concludes the proof.
\end{proof}

\subsection{Number of particles killed by the barrier}
\label{subsection:number-of-particles-killed-by-the-barrier}

\def\stopped{0}

We still consider here a BBM starting with a single particle at $x\ge0$ and where particles are killed when hitting 0. 
Let $\cL$ denote the stopping line of the killed particles and, for $u \in \cL$, let $\Delta_u$ be the killing time of $u$.
Our interest in this section is the number 
$$N_{[a,b]} \coloneqq \# \{ u \in \cL : \Delta_u \in [a,b] \}$$
of particles that are killed at 0 between times $a$ and $b$. 
More generally, we introduce as before a new random variable, where we remove the particles with too many children: for $\kappa \ge 0$ and $0\le a\le b$, we set
\begin{align*}
C_\kappa
&\coloneqq
\left\{
u \in \cL :
\forall v < u,  O_v \leq \kappa \e^{X_v(d_v)/2}
\right\}, \\
N_{[a,b],\kappa}
& \coloneqq 
\# \{ u \in C_\kappa : \Delta_u \in [a,b] \}.
\end{align*}

It will be most convenient to consider a different branching Markov process (in the sense of Hardy and Harris \cite{Hardy2009}), namely, BBM where particles are \emph{stopped} at the origin. Formally, this is a branching Markov process where the motion of the particles is standard Brownian motion on $[0,\infty)$, stopped when it hits the origin, and the branching rate is equal to $\lambda$ on $(0,\infty)$ and zero at the origin (the offspring distribution is the same as before). We will denote the law and expectation w.r.t.~this process by $\P^{\stopped}_x$ and $\E^{\stopped}_x$, respectively. It is clear that $(W_t)_{t\ge0}$ is still a martingale under $\P^{\stopped}_x$. As in Section~\ref{subsection:change-of-probabilities}, we can therefore define a new measure $\Q^{\stopped}_x$ by changing the measure $\P^{\stopped}_x$ w.r.t~ the normalized martingale $(\e^{-x} W_t)_{t\ge0}$. This measure has an analogous description as the measure $\Q_x$, except that all particles (including the spine), are stopped at the origin and do not branch once stopped. Furthermore, we have the following many-to-one lemma:

\begin{prop} \label{prop:change-of-measure-Q0}
Let $x\ge 0$.
Let $s \geq 0$ and let $(H_s(u))_{u\in \T}$ be a family of uniformly bounded $\sF_s$-measurable random variables. Then, we have
\begin{align*}
\Ecii{x}{\stopped}{\sum_{u\in\cN(s)} \e^{-X_u(s)} H_s(u)} = \e^{-x}\Eci{\Q^{\stopped}_x}{H_s(w_s)}.
\end{align*}
\end{prop}



As a first application of Proposition~\ref{prop:change-of-measure-Q0}, we calculate the first moment of $N_{[a,b]}$ (this calculation is standard but we include it for convenience).

\begin{lem} \label{lem:first-moment-N}
For $x \geq 0$ and $0\le a\le b$, we have
\[
\Eci{x}{N_{[a,b]}} 
= \e^{-x} \int_a^b \frac{x}{\sqrt{2\pi}s^{3/2}} \e^{-x^2/2s} \diff s.
\]
Moreover, it follows that $\E_x[N_{[0,\infty)}] = \e^{-x}$.
\end{lem}
\begin{proof}
We first rewrite the expectation using BBM stopped at the origin:
\[
 \Eci{x}{N_{[a,b]}} = \Ecii{x}{\stopped}{\sum_{u\in\cN(b)} \1_{X_u(b) = 0,\,X_u(r)> 0\,\forall r<a}}.
\]
Note that if $X_u(b) = 0$, then $1 = \e^{-X_u(b)}$. We can therefore apply the many-to-one formula (Proposition~\ref{prop:change-of-measure-Q0}) and get,
\begin{align*}
\Eci{x}{N_{[a,b]}} 
& = \e^{-x} \Q_x^{\stopped}\left(X_{w_b}(b) = 0,\,X_{w_b}(r)> 0\,\forall r<a \right)\\
& = \e^{-x} \Ppi{x}{\inf \{ r \geq 0 : B_r = 0 \} \in [a,b]},
\end{align*}
where $(B_r)_{r \geq 0}$ is a standard Brownian motion started at $x$ under $\P_x$.
The law of the hitting time of the origin of Brownian motion started at $x$ is known (see e.g.~Equation 1.2.0.2 of Borodin and Salminen \cite{borodinsalminen2002}) and implies the first statement of the lemma. The second statement follows by setting $a=0$, letting $b\to\infty$ and observing that the hitting time is finite almost surely (or by direct calculation, using for example the change of variables $v = x/\sqrt{s}$).
\end{proof}

\begin{lem} \label{lem:not-in-C}
There exists a decreasing function $h \colon \R_+ \to \R_+$ such that $\lim_{a \to \infty} h(a) = 0$ and, for any $x \ge 0$, $\kappa>1$,
\[
\Eci{x}{N_{[0,\infty)} - N_{[0,\infty),\kappa}} \leq \frac{h(\kappa)}{\log \kappa} \e^{-x}.
\]
\end{lem}
\begin{proof}
As in the beginning of the proof of Lemma~\ref{lem:first-moment-N}, for any $s \geq 0$, we have
\begin{align*}
\Eci{x}{N_{[0,s]} - N_{[0,s],\kappa}}
& = \e^{-x} \Q_x^{\stopped}\left(X_{w_s}(s) = 0,\,w_s \notin C_\kappa\right).
\end{align*}
Under $\P_x$, let $(B_r)_{r \geq 0}$ denote a Brownian motion and let $\widehat{L}$ be a random variable independent of $(B_r)_{r \geq 0}$ following the size-biased reproduction law. Furthermore, set  $\tau \coloneqq \inf \{ r \geq 0 : B_r = 0 \}$.
Then, using the spinal decomposition description under $\Q^0_x$, we obtain as in Lemma~\ref{lem:not-in-B},
\begin{align*}
\Q_x^{\stopped}\left(X_{w_s}(s) = 0,\,w_s \notin C_\kappa\right)
& \leq \Eci{x}{\1_{\tau \leq s}
	\int_0^\tau \Ppsqi{x}{\widehat{L} > \kappa \e^{B_t/2}}{(B_t)_{t \geq 0}}
	\mu_1 \lambda\diff t} \\
& \leq \mu_1 \lambda \Eci{x}{
	\Ecsqi{x}{\int_0^\tau \1_{B_t < 2 \log_+ (\widehat{L}/\kappa)}\diff t}{\widehat{L}}
	}\\
& \leq \mu_1 \lambda \Ec{\left( 2 \log_+ (\widehat{L}/\kappa) \right)^2} && \text{by \eqref{eq:time-killed-BM-under-a}}\\
& \leq 4 \mu_1 \lambda \widehat{\E} \left[
	\left( \log_+ (\widehat{L}/\kappa) \right)^2
	\1_{\widehat{L} \geq \kappa}
	\frac{\log_+ \widehat{L}}{\log_+ \kappa}
	\right].
\end{align*}
Setting $h(\kappa) \coloneqq 4 \mu_1 \lambda \E[ L (\log_+(L/\kappa))^2 \log_+ L]$, we have $\lim_{\kappa \to \infty} h(\kappa) = 0$ by dominated convergence (and using Assumption \eqref{hypothese 1}).
Letting $s \to \infty$, this proves the result.
\end{proof}

\begin{lem} \label{lem:second-moment-N}
There exists $C >0$ such that we have, for any $x \geq 0$ and $\kappa \ge 1$,
\begin{align*}
\Eci{x}{N_{[0,\infty),\kappa}^2}
\leq C \kappa \e^{-x}.
\end{align*}
\end{lem}

\begin{proof}
First note that 
\[
 \Eci{x}{N_{[0,\infty),\kappa}^2} = \Eci{x}{N_{[0,\infty),\kappa}(N_{[0,\infty),\kappa}-1)} + \Eci{x}{N_{[0,\infty),\kappa}},
\]
and the second summand is bounded by $\e^{-x}$ by the inequality $N_{[0,\infty),\kappa} \le N_{[0,\infty)}$ and Lemma~\ref{lem:first-moment-N}. It thus remains to show that the first summand is bounded by $C \kappa \e^{-x}$ for some $C>0$. 

The proof is similar to the proof of Lemma~\ref{lem:second-moment}. We first have for $s\ge0$, by Proposition \ref{prop:change-of-measure-Q0},
\begin{align*}
 \Eci{x}{N_{[0,s],\kappa}(N_{[0,s],\kappa}-1)}
&= \Ecii{x}{\stopped}{\sum_{u\in\cN(s)} \1_{X_u(s) = 0,\, u\in C_\kappa} (N_{[0,s],\kappa}-1)}\\
& = \e^{-x} 
\E_{\Q^{\stopped}_x} \left[ 
(N_{[0,s],\kappa}-1)
\1_{X_{w_s}(s) = 0,\,w_s \in C_\kappa}
\right].
\end{align*}
Denoting again $\Pi_s$ the set of branching times of the spine before $s$ and using the spinal decomposition description, we get 
\begin{align*}
& \E_{\Q^{\stopped}_x} \left[ 
(N_{[0,s],\kappa}-1)
\1_{X_{w_s}(s) = 0,\,w_s \in C_\kappa}
\right] \\
& = \E_{\Q^{\stopped}_x} \left[ 
\1_{X_{w_s}(s) = 0,\,w_s \in C_\kappa}
\sum_{t \in \Pi_s}
	\left( O_{w_t} - 1 \right)
	\Eci{X_{w_t}(t)}{N_{[0,s-t],\kappa}}
\right]\\
&\leq \kappa \E_{\Q^{\stopped}_x} \left[ 
\1_{X_{w_s}(s) = 0}
\sum_{t \in \Pi_s}
	\e^{X_{w_t}(t)/2}
	\Eci{X_{w_t}(t)}{N_{[0,\infty)}}
\right].
\end{align*}
We denote again by $(B_r)_{r \geq 0}$ a Brownian motion started at $x$ under
$\P_x$ and set $\tau \coloneqq \inf \{ r \geq 0 : B_r = 0 \}$. Using
Lemma~\ref{lem:first-moment-N} and the formula for
expectations of additive functionals of Poisson point processes, the last
estimates gives
\begin{align*}
 \Eci{x}{N_{[0,s],\kappa}(N_{[0,s],\kappa}-1)} \le \kappa\mu_1\lambda \e^{-x} \Eci{x}{\1_{\tau\le s} \int_0^\tau \e^{-B_t/2}\diff t},
\end{align*}
and, letting $s\to\infty$, we get
\[
  \Eci{x}{N_{[0,\infty),\kappa}(N_{[0,\infty),\kappa}-1)} \le \kappa\mu_1\lambda \e^{-x}  \Eci{x}{\int_0^\tau \e^{-B_t/2}\diff t}.
\]
This last integral is bounded by a constant by \eqref{eq:integral-killed-BM}, which finishes
the proof of the theorem.
%
\end{proof}

\subsection{Truncation of the critical additive martingale}
\label{subsection:truncation-critical-additive-martingale}

In this section, using again the new probability measure $\widetilde{\Q}_x$, we study the truncation of the critical additive martingale, when particles are killed below 0, defined by
\begin{align*}
\widetilde{W}_s
\coloneqq 
\sum_{u \in \cN(s)} \e^{-X_u(s)} \1_{\forall r \in [0,s], X_u(r) > 0},
\quad s \geq 0.
\end{align*}
Note that $(\widetilde{W}_s)_{s \geq 0}$ is not a martingale.
Moreover, we introduce, for $\kappa\geq 0$ and $s \geq 0$, 
\begin{align*}
\widetilde{W}_{s,\kappa} 
& \coloneqq 
\sum_{u \in \cN(s)} \e^{-X_u(s)} 
\1_{\forall r \in [0,s], X_u(r) > 0} \1_{u \in B_{s,\kappa}},
\end{align*}
where $B_{s,\kappa}$ has been introduced in Subsection
\ref{subsection:truncated-derivative-martingale}.
\begin{lem} \label{lem:control-first-moment-W}
For $x \ge 0$, set $F(x) \coloneqq \sqrt{2/\pi} \int_0^x \e^{-z^2/2} \diff
z$.
For any $x,s\ge0$ we have
\[
\Eci{x}{\widetilde{W}_s} 
= \e^{-x} F \left( \frac{x}{\sqrt{s}} \right)
\leq \e^{-x} \left( 1 \wedge \sqrt{\frac{2}{\pi}} \frac{x}{\sqrt{s}} \right).
\]
\end{lem}
\begin{proof}
Using the many-to-one formula (Proposition~\ref{prop:change-of-measure-Q}) and
the reflection principle, we have
\begin{align*}
\Eci{x}{\widetilde{W}_s} 
= \e^{-x} \Ppi{x}{\forall r \in [0,s], B_r \geq 0} 
= \e^{-x} \Pp{\abs{B_s} \leq x}
= \e^{-x} F \left( \frac{x}{\sqrt{s}} \right),
\end{align*}
because $F(y) = \Pp{\abs{B_1} \leq y}$.
Then, note that, for any $y > 0$, we have $F(y) \leq \sqrt{2/\pi} y$ and also
$F(y) \leq 1$.
\end{proof}
\begin{lem} \label{lem:not-in-B-for-W}
There exists a decreasing function $h \colon \R_+ \to \R_+$ such that 
$\lim_{a \to \infty} h(a) = 0$ and, for any $x,\kappa,s \ge 0$,
\[
\Eci{x}{\widetilde{W}_s - \widetilde{W}_{s,\kappa}} 
\leq h(\kappa) \e^{-x} 
	\left( \frac{1}{\sqrt{s}} + \frac{x}{s} \right).
\]
\end{lem}
\begin{proof}
Using Proposition \ref{prop:change-of-measure-Qtilde}, we get
\[
\Eci{x}{\widetilde{W}_s - \widetilde{W}_{s,\kappa}} 
= \Eci{x}{\sum_{u \in \cN(s)} X_u(s)\e^{-X_u(s)} 
	\1_{\forall r \in [0,s], X_u(r) > 0} \frac{\1_{u \notin
B_{s,\kappa}}}{X_u(s)}}
= x \e^{-x} \Eci{\widetilde{\Q}_x}{\frac{\1_{w_s \notin
B_{s,\kappa}}}{X_{w_s}(s)}}.
\]
Then, using the spinal decomposition description, we obtain
\begin{align*}
\Eci{\widetilde{\Q}_x}
	{\frac{\1_{w_s \notin B_{s,\kappa}}}{X_{w_s}(s)}}
& \leq \Eci{x}{
	\frac{1}{R_s}
	\int_0^s 
	\Ppsq{\widehat{L} > \kappa \e^{R_r/2}}{(R_r)_{r\geq 0}}
	\mu_1 \lambda\diff r} \\
& = \mu_1 \lambda \Eci{x}{ \Ecsqi{x}{
	\frac{1}{R_s}
	\int_0^s \1_{R_r < 2 \log_+ (\widehat{L}/\kappa)}\diff r}{\widehat{L}} }
\\
& \leq \mu_1 \lambda \Ec{C \left(2 \log_+ \left(\widehat{L}/\kappa \right)
\right)^3 
	\left( \frac{1}{x \sqrt{s}} + \frac{1}{s} \right)},
\end{align*}
by applying Lemma \ref{lem:time-bessel-under-a-and-inverse}. 
Setting $h(\kappa) \coloneqq C \mu_1\lambda \E[ L (\log_+(L/\kappa))^3]$, it proves the
result. 
\end{proof}
\begin{lem} \label{lem:second-moment-for-W}
There exists $C >0$ such that, for any $x \geq 0$, $\kappa \ge 1$ and $s \geq 2$, we have
\begin{align*}
\Eci{x}{\widetilde{W}_{s,\kappa}^2}
\leq C \kappa \e^{-x} \left( \frac{1}{s} + \frac{x \ln s}{s^{3/2}} \right).
\end{align*}
\end{lem}

\begin{proof}
Using Proposition \ref{prop:change-of-measure-Qtilde}, we have 
\begin{align*}
\Eci{x}{\widetilde{W}_{s,\kappa}^2}
= x \e^{-x} \Eci{\widetilde{\Q}_x}{\widetilde{W}_{s,\kappa} 
	\frac{\1_{w_s \in B_{s,\kappa}}}{X_{w_s}(s)}}
\leq x \e^{-x} \Eci{\widetilde{\Q}_x}{\widetilde{W}_s 
	\frac{\1_{w_s \in B_{s,\kappa}}}{X_{w_s}(s)}}.
\end{align*}
As in the proof of Lemma~\ref{lem:second-moment}, we get 
\begin{align*}
\Eci{\widetilde{\Q}_x}{\widetilde{W}_s 
	\frac{\1_{w_s \in B_{s,\kappa}}}{X_{w_s}(s)}}
& = \Eci{\widetilde{\Q}_x}{
	\frac{\1_{w_s \in B_{s,\kappa}}}{X_{w_s}(s)}
\left(
	\sum_{r \in \Pi_s}
	\left( O_{w_r} - 1 \right)
	\Eci{X_{w_r}(r)}{\widetilde{W}_{s-r}} + \e^{-X_{w_s}(s)} 
\right)
} \\
&\le \Eci{x}{\frac{1}{R_s} \left(\int_0^s \kappa \e^{R_r/2}
	\Eci{R_r}{\widetilde{W}_{s-r}} \mu_1 \lambda \diff r + \e^{-R_s}\right)}
\\
& \le \mu_1 \lambda\kappa 
	\Eci{x}{\frac{1}{R_s} \int_0^s \e^{-R_r/2} 
	\left( 1 \wedge \sqrt{\frac{2}{\pi}} \frac{R_r}{\sqrt{s-r}} \right)
	\diff r} + \Eci{x}{\frac{1}{R_s}\e^{-R_s}},
\end{align*}
by applying Lemma \ref{lem:control-first-moment-W}. The first expectation is
bounded by $C \left( \frac{1}{xs} + \frac{\ln s}{s^{3/2}}\right)$ for some
constant $C$, by Lemma \ref{lem:integral-bessel-and-inverse}, for $x>0$ (and
the statement of the proposition is trivial if $x=0$). The second expectation is
bounded by $C/s^{3/2}$, because the density of $R_s$ under $\P_x$ is bounded by
$Cy^2/s^{3/2}$ for all $x\ge0$, by \eqref{eq:density_R} and the inequality
$1-\e^{-x} \le x$ for all $x\ge0$. Together with the previous inequalities, this
yields the result.
\end{proof}

\section{The particles staying above \texorpdfstring{$\gamma_t$}{gamma\_t}: proof of Proposition~\ref{prop:control-Ztilde-infty}}
\label{section:control-without-fluctuations}

In this section, we prove Proposition \ref{prop:control-Ztilde-infty} using the results from Section~\ref{subsection:truncated-derivative-martingale}.
As explained below, it will be enough to consider $a=1$.
We recall the notation from Section~\ref{section:proof-of-theorem}: for $t \geq 0$,
\begin{align*}
\gamma_t &= \frac 1 2 \log t + \beta_t,\quad \text{with }\beta_t \to \infty
 \text{ and } 
\frac{\beta_t}{t^{1/4}} \to 0\text{ as $t\to\infty$,}\\
\widetilde{Z}_s^{t,\gamma_t} 
& = \sum_{u \in \cN(s)} (X_u(s)-\gamma_t) \e^{-X_u(s)} 
\1_{\forall r \in [t,s], X_u(r) > \gamma_t}, 
\quad s \geq t.
\end{align*}
We start with a preliminary lemma.
\begin{lem} \label{lem:concentration-Ztilde}
We have the following convergence in probability
\begin{align*}
\sqrt t \abs{\widetilde{Z}_\infty^{t,\gamma_t} - \widetilde{Z}_t^{t,\gamma_t}} \xrightarrow[t\to\infty]{} 0.
\end{align*}
\end{lem}

\begin{proof}[Proof of Lemma~\ref{lem:concentration-Ztilde}]
The proof basically follows a first- and second moment argument, with a truncation as in Section~\ref{subsection:truncated-derivative-martingale} in order to handle offspring distributions of infinite variance. By analogy with Section~\ref{subsection:truncated-derivative-martingale}, we set, for $\kappa \geq 1$ and $s \geq t$,
\begin{align*}
B_{s,\kappa}^{t,\gamma_t}
&\coloneqq
\left\{
u \in \cN(s) :
\forall v < u, d_v \geq t \Rightarrow O_v \leq \kappa \e^{(X_v(d_v)-\gamma_t)/2}
\right\} \\
\widetilde{Z}_{s,\kappa}^{t,\gamma_t}
& \coloneqq 
\sum_{u \in \cN(s)} (X_u(s)-\gamma_t) \e^{-X_u(s)} 
\1_{\forall r \in [t,s], X_u(r) > \gamma_t} \1_{u \in B_{s,\kappa}^{t,\gamma_t}}.
\end{align*}
It is enough to show the following: for every $\varepsilon>0$,
\begin{equation}
 \label{eq:enoughtoshow}
\Ppsq{\abs{\widetilde{Z}_s^{t,\gamma_t} - \widetilde{Z}_t^{t,\gamma_t}} 
	\geq \frac{\varepsilon}{\sqrt{t}}}{\sF_t} \xrightarrow[t\to\infty]{} 0,\quad\text{in probability.}
\end{equation}
Noting that $\widetilde{Z}_s^{t,\gamma_t} - \widetilde{Z}_{s,\kappa}^{t,\gamma_t}\geq 0$ and $\E[\widetilde{Z}_s^{t,\gamma_t}| \sF_t] = \widetilde{Z}_t^{t,\gamma_t}$, we decompose
\begin{align}
& \Ppsq{\abs{\widetilde{Z}_s^{t,\gamma_t} - \widetilde{Z}_t^{t,\gamma_t}} 
	\geq \frac{3\varepsilon}{\sqrt{t}}}{\sF_t} \nonumber \\
& \leq \Ppsq{\widetilde{Z}_s^{t,\gamma_t} - \widetilde{Z}_{s,\kappa}^{t,\gamma_t} 
	\geq \frac{\varepsilon}{\sqrt{t}}}{\sF_t}
+ \Ppsq{\abs{\widetilde{Z}_{s,\kappa}^{t,\gamma_t} - \Ecsq{\widetilde{Z}_{s,\kappa}^{t,\gamma_t}}{\sF_t}} 
	\geq \frac{\varepsilon}{\sqrt{t}}}{\sF_t} 
+ \1_{\Ecsq{\widetilde{Z}_s^{t,\gamma_t} - \widetilde{Z}_{s,\kappa}^{t,\gamma_t}}{\sF_t} 
	\geq \varepsilon/\sqrt{t}} \nonumber \\
& \leq 2 \cdot \frac{\sqrt{t}}{\varepsilon} 
	\Ecsq{\widetilde{Z}_s^{t,\gamma_t} - \widetilde{Z}_{s,\kappa}^{t,\gamma_t}}{\sF_t} 
+ \left( \frac{\sqrt{t}}{\varepsilon} \right)^2 \Varsq{\widetilde{Z}_{s,\kappa}^{t,\gamma_t}}{\sF_t},
\label{aa}
\end{align}
using Markov's inequality for the first and third terms and Chebychev's inequality for the second one.
For the first term in \eqref{aa}, using the branching property at time $t$ and Lemma \ref{lem:not-in-B}, we have with the notation used there,
\begin{align}
\Ecsq{\widetilde{Z}_s^{t,\gamma_t} - \widetilde{Z}_{s,\kappa}^{t,\gamma_t}}{\sF_t}
& = \sum_{u \in \cN(t)} \e^{-\gamma_t} 
	\Eci{X_u(t)-\gamma_t}{\widetilde{Z}_{s-t} - \widetilde{Z}_{s-t,\kappa}} 
	\1_{X_u(t) > \gamma_t} \nonumber \\
& \leq \sum_{u \in \cN(t)} \e^{-\gamma_t} 
C\e^{-(X_u(t)-\gamma_t)} h(\kappa)
= C W_t h(\kappa). \label{ab}
\end{align}
For the second term in \eqref{aa}, using again the branching property at time $t$ and Lemma \ref{lem:second-moment}, we get
\begin{align}
\Varsq{\widetilde{Z}_{s,\kappa}^{t,\gamma_t}}{\sF_t}
& \leq \sum_{u \in \cN(t)}  
	\Eci{X_u(t)-\gamma_t}{ \left(\e^{-\gamma_t} \widetilde{Z}_{s-t,\kappa} \right)^2}
	\1_{X_u(t) > \gamma_t} \nonumber \\
& \leq \sum_{u \in \cN(t)} \e^{-2\gamma_t} C \kappa \e^{-(X_u(t)-\gamma_t)}
= C \kappa \e^{-\gamma_t} W_t. \label{ac}
\end{align}
Combining \eqref{aa}, \eqref{ab}, \eqref{ac} and taking $s \to \infty$, we get for every $\kappa \geq 1$ and $t \geq 0$, 
\begin{align}
\Ppsq{\abs{\widetilde{Z}_\infty^{t,\gamma_t} - \widetilde{Z}_t^{t,\gamma_t}} 
	\geq \frac{\varepsilon}{\sqrt{t}}}{\sF_t}
& \leq C \sqrt{t} W_t \left(
\frac{h(\kappa)}{\varepsilon}
+ \frac{\kappa \e^{-\beta_t}}{\varepsilon^2}
\right). \label{ad}
\end{align}
Now recall from \eqref{eq:convergence-W_t} that $\sqrt t W_t$ converges in probability, as $t\to\infty$.
Taking $\kappa = \e^{\beta_t/2}$ and using that $h(\kappa) \to 0$ as $\kappa \to \infty$, the right-hand side of \eqref{ad} then goes to 0 in probability as $t \to \infty$, which shows \eqref{eq:enoughtoshow}.
The result follows.
\end{proof}
The next lemma shows that the contribution of the particles below $\gamma_t$ at time $t$ is negligible. It is stated in greater generality since it will be used in Section~\ref{section:killed-particles} as well.
It is here that we use the assumption $\gamma_t = o(t^{1/4})$. 
Define for $s\geq t\geq0$,
\begin{align}
\widetilde W_s^{t,\gamma_t}
\coloneqq \sum_{u \in \cN(t)} \e^{-X_u(t)} \1_{\forall r\in[t,s], X_u(r) > \gamma_t}.
\label{initial}
\end{align}
\begin{lem} \label{lem:cv-for-L_init}
For all $a\ge 1$, we have $\gamma_t \sqrt{t} (W_{at} - \widetilde W_{at}^{t,\gamma_t}) \to 0$ in probability as $t\to\infty$.
\end{lem}
\begin{proof} Fix $\varepsilon >0$.
By \eqref{eq:minimum}, there exists $L >0$ such that $\P (\min_{s \geq 0} \min_{u\in\cN(s)} X_u(s) \leq -L) \leq \varepsilon$. By Markov's inequality,
\begin{align} 
\P \left(W_{at} - \widetilde W_{at}^{t,\gamma_t} \geq \frac{\varepsilon}{\gamma_t \sqrt{t}} \right)
& \leq \varepsilon 
+ \varepsilon^{-1}\gamma_t\sqrt{t}\,
	\Ec{(W_{at} - \widetilde W_{at}^{t,\gamma_t})
	\1_{\min_{s \geq 0} \min_{u\in\cN(s)} X_u(s) > -L}} \nonumber \\
& \leq \varepsilon 
+  \varepsilon^{-1}\gamma_t\sqrt{t}\,
	\E \Biggl[ 
		\sum_{u \in \cN(t)} \e^{-X_u(at)} \1_{\min_{r\in[t,at]} X_u(r) \leq \gamma_t}
		\1_{\forall s \in [0,at], X_u(s) > -L}
		\Biggr]. \label{dx}
\end{align}
It follows from the many-to-one formula that the expectation on the right-hand side of \eqref{dx} is equal to
\begin{align*}
\Eci{L}{\1_{\min_{r\in[t,at]} B_r \leq \gamma_t+L} \1_{\min_{s \in [0,at]} B_s > 0}}
& = L \Eci{L}{ \frac{1}{R_t} \1_{\min_{r\in[t,at]} R_r \leq \gamma_t+L}},
\end{align*}
using \eqref{eq:link-between-R-and-B}.
By Lemma~\ref{lem:min-bessel-and-inverse},
\[
\Eci{L}{ \frac{1}{R_t} \1_{\min_{r\in[t,at]} R_r \leq \gamma_t+L}} \le C \left(\frac{(\gamma_t+L)^2}{t^{3/2}} +  \frac{\gamma_t+L}{\sqrt{a-1}\,t} \1_{a>1}\right).
\]
Coming back to \eqref{dx}, we proved that 
\begin{align*}
\P \left(W_{at} - \widetilde W_{at}^{t,\gamma_t} \geq \frac{\varepsilon}{\gamma_t \sqrt{t}} \right)
& \leq \varepsilon 
+ C \varepsilon^{-1}L\left(
	 \frac{\gamma_t (\gamma_t + L)^2}{t} +
	\frac{\gamma_t (\gamma_t + L)}{\sqrt{(a-1)t}}\1_{a>1}\right) \\
& \xrightarrow[t\to\infty]{} \varepsilon,
\end{align*}
because $\gamma_t = o(t^{1/4})$.
This proves the lemma.
\end{proof}
\begin{proof}[Proof of Proposition \ref{prop:control-Ztilde-infty}]
First note that it is sufficient to prove the case $a=1$: for $a > 1$, $\beta'_{at} \coloneqq \gamma_t - \frac{1}{2} \log (at)$ satisfies \eqref{eq:assumption-beta_t} with $at$ instead of $t$, so one can directly apply the case $a=1$.
We therefore have to show:
\begin{equation}
 \label{eq:havetoshow}
\sqrt{t} 
\abs{\widetilde{Z}_\infty^{t,\gamma_t} - \left(Z_t - \gamma_t W_t \right)} \xrightarrow[t\to\infty]{\P} 0.
\end{equation}
We decompose:
\[
 \sqrt{t} 
\abs{\widetilde{Z}_\infty^{t,\gamma_t} - \left(Z_t - \gamma_t W_t \right)} \le \sqrt{t} 
\abs{\widetilde{Z}_\infty^{t,\gamma_t} - \widetilde{Z}_t^{t,\gamma_t}} +  \sqrt{t} 
\abs{\widetilde{Z}_t^{t,\gamma_t} - \left(Z_t - \gamma_t W_t \right)}.
\]
The first term goes to 0 in probability by Lemma~\ref{lem:concentration-Ztilde}. The second term can be written as
\begin{align*}
\sqrt{t} 
\abs{\widetilde{Z}_t^{t,\gamma_t} - \left(Z_t - \gamma_t W_t \right)}
& = \sqrt{t} \sum_{u \in \cN(t)} (\gamma_t-X_u(t)) \e^{-X_u(t)}
\1_{X_u(t) \leq \gamma_t}.
\end{align*}
Since $\min_{u\in\cN(t)} X_u(t) \to \infty$ a.s. as $t\to\infty$ by \eqref{eq:minimum}, the right-hand side is bounded by $\gamma_t \sqrt{t} (W_{t} - \widetilde W_{t}^{t,\gamma_t})$ with high probability for large $t$, and this goes to 0 by Lemma~\ref{lem:cv-for-L_init} as $t\to\infty$. Equation~\eqref{eq:havetoshow} follows.
\end{proof}

\section{The particles going below \texorpdfstring{$\gamma_t$}{gamma\_t}: proof of Propositions~\ref{prop:control-fluctuations} \& \ref{prop:control-F^t_bad}}
\label{section:killed-particles}

In this section we prove Propositions~\ref{prop:control-fluctuations} and \ref{prop:control-F^t_bad}. We use throughout the notation from Section~\ref{section:proof-of-theorem}, in particular the stopping line $\cL^{at,\gamma_t}$ and its decomposition into
\[
 \cL^{at,\gamma_t} = \cL^{at,\gamma_t}_\mathrm{good} \cup \cL^{at,\gamma_t}_\text{bad}.
\]
Section~\ref{subsection:bad_particles} handles the ``bad'' particles: Proposition~\ref{prop:control-F^t_bad} is proven there. Section~\ref{subsection:number_of_good_particles} contains bounds on the number of good particles hitting $\gamma_t$ at a certain time. Section~\ref{subsection:contribution_of_good_particles} wraps up the proof of Proposition~\ref{prop:control-fluctuations}.

\subsection{The ``bad'' particles}
\label{subsection:bad_particles}

In this section, we prove that the ``bad'' particles have a negligible contribution to $Z_\infty$, i.e., we prove Proposition~\ref{prop:control-F^t_bad}. We will need the following estimate: there exists a constant $C>0$, such that for all $x\ge0$,
\begin{equation}
\label{eq:tail-Z_infty-3}
 \E[Z_\infty \wedge x] \le C(1+\log_+ x).
\end{equation}
This estimate is a direct consequence of \eqref{eq:tail-Z_infty} and the equality $\E[X\wedge x] = \int_0^x\P(X>x) \diff x$ for a non-negative random variable $X$.
\begin{proof}[Proof of Proposition \ref{prop:control-F^t_bad}]
We perform a truncated first moment computation. For simplicity, write $\Delta_u$ for $\Delta_u^{at,\gamma_t}$. Then note that
\[
F^{at,\gamma_t}_\text{bad}\wedge 1 
\le 
\sum_{u \in \cL^{at,\gamma_t}_\text{bad}} \e^{-X_u(\Delta_u)} \left(Z_\infty^{(u,at,\gamma_t)} \wedge \e^{X_u(\Delta_u)}\right) 
\le  
\sum_{u \in \cL^{at,\gamma_t}_\text{bad}} \e^{-X_u(\Delta_u)}\left(Z_\infty^{(u,at,\gamma_t)} \wedge \e^{\gamma_t}\right),
\]
because the particles $u\in \cL^{at,\gamma_t}_\text{bad}$ satisfy $X_u(\Delta_u) \le \gamma_t$. This gives
\begin{align}
\Ecsq{F^{at,\gamma_t}_\text{bad}\wedge 1}{\sF_{\cL^{at,\gamma_t}}}
& \leq \sum_{u \in \cL^{at,\gamma_t}_\text{bad}} \e^{-X_u(\Delta_u)}
	\Ecsq{Z_\infty^{(u,at,\gamma_t)}\wedge \e^{\gamma_t}}{\sF_{\cL^{at,\gamma_t}}}
	\nonumber \\
& \leq C (1+ \gamma_t) \sum_{u \in \cL^{at,\gamma_t}_\text{bad}} \e^{-X_u(\Delta_u)}, \nonumber
\end{align}
by \eqref{eq:tail-Z_infty-3}. Now define $\cN_\text{bad}(at)$ as the set of particles $u \in \cN(at)$ such that $\min_{r\in[t,at]} X_u(r) \leq \gamma_t$. Then, $\cL^{at,\gamma_t}_\text{bad}$ forms exactly the descendants of the particles in $\cN_\text{bad}(at)$ at the time when they go below $\gamma_t$. 
By the branching property at time $at$ and Lemma \ref{lem:first-moment-N} (for particles above $\gamma_t$ at time $at$),
\[
\Ecsq{\sum_{u \in \cL^{at,\gamma_t}_\text{bad}} \e^{-X_u(\Delta_u)}}{\sF_{at}} = \sum_{v\in\cN_\text{bad}(at)} \e^{-X_v(at)} = W_{at} - \widetilde W_{at}^{t,\gamma_t},
\]
where $\widetilde W_{at}^{t,\gamma_t}$ is defined in Section~\ref{section:control-without-fluctuations}.
Altogether, we get
\[
\Ecsq{\sqrt{t}(F^{at,\gamma_t}_\text{bad}\wedge 1)}{\sF_{at}} \le C(1+\gamma_t)\sqrt t (W_{at} - \widetilde W_{at}^{t,\gamma_t}).
\]
The lemma now readily follows from Lemma~\ref{lem:cv-for-L_init} and Markov's inequality.
\end{proof}

\subsection{The number of ``good'' particles}
\label{subsection:number_of_good_particles}

We recall that $\cL^{at,\gamma_t}_\mathrm{good}$ is a decreasing set in $a$, with $\cL^{at,\gamma_t}_\mathrm{good} = \{ u \in \cL^{t,\gamma_t}_\mathrm{good} : \Delta_u^{t,\gamma_t} > at\}$. 
For each $a \geq 1$, we set 
\begin{align*}
N^{at,\gamma_t}_\mathrm{good} \coloneqq \# \cL^{at,\gamma_t}_\mathrm{good}.
\end{align*} 
In this section, we study the convergence of $N^{at,\gamma_t}_\mathrm{good}$ after a proper renormalization.
\begin{lem} \label{lem:concentration-N}
For any $a \ge 1$, we have the following convergence in probability
\begin{align*}
\beta_t \e^{-\beta_t}\abs{N^{at,\gamma_t}_\mathrm{good} - \Ecsq{N^{at,\gamma_t}_\mathrm{good}}{\sF_{at}}}
& \xrightarrow[t\to\infty]{} 0.
\end{align*}
\end{lem}
\begin{proof}
The proof goes by a truncated second moment computation as in Section~\ref{subsection:number-of-particles-killed-by-the-barrier}.
Fix $a\ge 1$ and $\varepsilon > 0$. We set, for $\kappa > 1$, 
\begin{align*}
C_\kappa^{at,\gamma_t}
& \coloneqq
\left\{
u \in \cL^{at,\gamma_t}_\mathrm{good} :
\forall v < u, d_v \geq at \Rightarrow O_v \leq \kappa \e^{(X_v(d_v)-\gamma_t)/2}
\right\} \\
N^{at,\gamma_t}_{\mathrm{good},\kappa} 
& \coloneqq \# C_\kappa^{at,\gamma_t}.
\end{align*}
Noting that $N^{at,\gamma_t}_\mathrm{good} - N^{at,\gamma_t}_{\mathrm{good},\kappa}\geq 0$ and working conditionally on $\sF_{at}$, we have
\begin{align}
& \Ppsq{\abs{N^{at,\gamma_t}_\mathrm{good} - \Ecsq{N^{at,\gamma_t}_\mathrm{good}}{\sF_{at}}} 
	\geq \frac{3\varepsilon\e^{\beta_t}}{\beta_t}}{\sF_{at}} \nonumber \\
\begin{split}
& \leq \Ppsq{N^{at,\gamma_t}_\mathrm{good} - N^{at,\gamma_t}_{\mathrm{good},\kappa}
	\geq \frac{\varepsilon \e^{\beta_t}}{\beta_t}}{\sF_{at}} 
+ \Ppsq{\abs{N^{at,\gamma_t}_{\mathrm{good},\kappa} - \Ecsq{N^{at,\gamma_t}_{\mathrm{good},\kappa}}{\sF_{at}}} 
	\geq \frac{\varepsilon \e^{\beta_t}}{\beta_t}}{\sF_{at}} \\
& \relphantom{\leq} {} 
+ \1_{\E \bigl[ N^{at,\gamma_t}_\mathrm{good} - N^{at,\gamma_t}_{\mathrm{good},\kappa} \big| \sF_{at} \bigr] \geq \varepsilon \e^{\beta_t}/\beta_t}
\end{split} \nonumber \\
& \leq  \frac{2\beta_t}{\varepsilon \e^{\beta_t}} 
	\Ecsq{N^{at,\gamma_t}_\mathrm{good} - N^{at,\gamma_t}_{\mathrm{good},\kappa}}{\sF_{at}} 
+ \left( \frac{\beta_t}{\varepsilon \e^{\beta_t}} \right)^2 
	\Varsq{N^{at,\gamma_t}_{\mathrm{good},\kappa}}{\sF_{at}}, \label{da}
\end{align}
using Markov's inequality for the first and third terms and Chebychev's inequality for the second one.

For the first term in \eqref{da}, using the branching property at time $at$ and Lemma \ref{lem:not-in-C}, we have with the notation used there,
\begin{align}
\Ecsq{N^{at,\gamma_t}_\mathrm{good} - N^{at,\gamma_t}_{\mathrm{good},\kappa}}{\sF_{at}} 
& = \sum_{u \in \cN(at)} 
	\Eci{X_u(at)-\gamma_t}{N_{(0,\infty)} - N_{(0,\infty),\kappa}} 
	\1_{X_u(at) > \gamma_t} \nonumber \\
& \leq \sum_{u \in \cN(at)} \e^{-(X_u(at)-\gamma_t)} \frac{h(\kappa)}{\log \kappa} 
= \sqrt{t} \e^{\beta_t} W_t \frac{h(\kappa)}{\log \kappa}, \label{db}
\end{align}
For the second term in \eqref{aa}, using again the branching property at time $at$ and Lemma \ref{lem:second-moment-N}, we get
\begin{align}
\Varsq{N^{at,\gamma_t}_{\mathrm{good},\kappa}}{\sF_{at}}
& \leq \sum_{u \in \cN(at)}  
	\Eci{X_u(at)-\gamma_t}{ \left(N_{(0,\infty),\kappa} \right)^2}
	\1_{X_u(at) > \gamma_t} \nonumber \\
& \leq \sum_{u \in \cN(at)} C \kappa \e^{-(X_u(at)-\gamma_t)}
= C \kappa \sqrt{t} \e^{\beta_t} W_{at}, \label{dc}
\end{align}
Combining \eqref{da}, \eqref{db} and \eqref{dc} with $\kappa = \e^{\beta_t/2}$, it follows that
\begin{align*}
\Ppsq{\e^{-\beta_t} \abs{N^{at,\gamma_t}_\mathrm{good} - \Ecsq{N^{at,\gamma_t}_\mathrm{good}}{\sF_{at}}} 
	\geq \frac{\varepsilon}{\beta_t}}{\sF_{at}} 
& \leq C \sqrt{t} W_{at} \left( 
\frac{h(\e^{\beta_t/2})}{\varepsilon} 
+ \frac{\beta_t^2 \e^{-\beta_t/2}}{\varepsilon^2}
\right)
\xrightarrow[t\to\infty]{} 0,
\end{align*}
in probability, using the convergence in probability of $\sqrt tW_t$ from \eqref{eq:convergence-W_t} and that $h(\e^{\beta_t/2}) \to 0$. The result follows.
\end{proof}
\begin{prop} \label{prop:control-number-of-killed-particles}
For every $a \geq1$, we have the following convergence in probability
\begin{align*}
\beta_t \abs{\e^{-\beta_t} N^{at,\gamma_t}_\mathrm{good} - \sqrt{t} W_{a t}}
\xrightarrow[t\to\infty]{} 0.
\end{align*}
\end{prop}
\begin{proof}
Using the branching property at time $at$, we have
\begin{align} \label{eq:first-moment-N_(1,infty)}
\e^{-\beta_t} \Ecsq{N^{at,\gamma_t}_\mathrm{good}}{\sF_{at}}
& = \e^{-\beta_t} \sum_{u \in \cN(at)} \1_{\forall r\in[t,at]:X_u(r) > \gamma_t}
	\Eci{X_u(at) - \gamma_t}{N_{(0,\infty)}}
= \sqrt{t} \widetilde{W}_{at}^{t,\gamma_t},
\end{align}
applying Lemma \ref{lem:first-moment-N}. Using the triangle inequality, it follows that
\[
 \beta_t \abs{\e^{-\beta_t} N^{at,\gamma_t}_\mathrm{good} - \sqrt{t} W_{a t}} \le \beta_t\e^{-\beta_t} \abs{N^{at,\gamma_t}_\mathrm{good} - \Ecsq{N^{at,\gamma_t}_\mathrm{good}}{\sF_{at}}} + \beta_t\sqrt{t}\abs{\widetilde{W}_{at}^{at,\gamma_t} - W_{at}}.
\]
Both terms vanish in probability as $t\to\infty$ by Lemmas \ref{lem:concentration-N} and \ref{lem:cv-for-L_init}. This concludes the proof.
%
\end{proof}
\begin{cor} \label{cor:cv-N}
Let $n \geq 1$, $1 \leq a_1 < \dots < a_n < a_{n+1} = \infty$ and $(z^t)_{t\geq 0} \in (\C^n)^{\R_+}$. 
Assume that, for any $1 \leq k \leq n$ and $t \geq 0$, $\mathrm{Re} (z_k^t) \leq 0$ and that $z^t$ converges to some $z \in \C^n$.
Then, we have the following convergence in probability
\begin{align*}
\Ecsq{\exp \left( \sum_{k=1}^n z_k^t 
	\e^{-\beta_t}  \left(N^{a_kt,\gamma_t}_\mathrm{good} - N^{a_{k+1}t,\gamma_t}_\mathrm{good}\right) 
	\right)}{\sF_t}
& \xrightarrow[t\to\infty]{} 
\exp \left( \sqrt{\frac{2}{\pi}} Z_\infty 
\sum_{k=1}^n z_k \left( \frac{1}{\sqrt{a_k}} - \frac{1}{\sqrt{a_{k+1}}} \right) 
\right).
\end{align*}
\end{cor}
\begin{proof}
Recalling that $\sqrt{t} W_{at} \to \sqrt{2/\pi a} Z_\infty$ in probability by \eqref{eq:convergence-W_t} and $Z_t \to Z_\infty$ a.s., it follows from Proposition \ref{prop:control-number-of-killed-particles} that, for every $1 \leq a < b \leq \infty$, we have
\begin{align*}
\abs{\e^{-\beta_t}  \left(N^{at,\gamma_t}_\mathrm{good} - N^{bt,\gamma_t}_\mathrm{good}\right)
- \sqrt{\frac{2}{\pi}} Z_t \left( \frac{1}{\sqrt{a}} - \frac{1}{\sqrt{b}} \right)}
\xrightarrow[t\to\infty]{\P} 0.
\end{align*}
Noting that everything is bounded by 1 because $\mathrm{Re} (z_k^t) \leq 0$ and using Remark \ref{rem:weak-convergence-in-proba}, it is sufficient to prove that 
\begin{align*}
\Ecsq{\exp \left( \sum_{k=1}^n z_k^t 
	\sqrt{\frac{2}{\pi}} Z_t
	\left( \frac{1}{\sqrt{a_k}} - \frac{1}{\sqrt{a_{k+1}}} \right) 
	\right)}{\sF_t}
& \xrightarrow[t\to\infty]{\P} 
\exp \left( \sqrt{\frac{2}{\pi}} Z_\infty
\sum_{k=1}^n z_k \left( \frac{1}{\sqrt{a_k}} - \frac{1}{\sqrt{a_{k+1}}} \right)  \right).
\end{align*}
But this holds because $Z_t$ is $\sF_t$-measurable and $Z_t \to Z_\infty$ a.s.
\end{proof}

\subsection{Contribution of the ``good'' particles}
\label{subsection:contribution_of_good_particles}

In this section, we use the results of the last section for proving Proposition \ref{prop:control-fluctuations}.
\begin{proof}[Proof of Proposition \ref{prop:control-fluctuations}] 
Note that we can assume w.l.o.g.\@ that $a_1 < a_2 < \dots < a_n$.
Using Proposition \ref{prop:control-number-of-killed-particles} and Remark \ref{rem:weak-convergence-in-proba}, it is sufficient to prove that 
\begin{align}
 &\Ecsq{\exp \left( i \sum_{k=1}^n \lambda_k \left( 
	\sqrt{t} F^{a_k t,\gamma_t}_\mathrm{good} 
	- \beta_t \e^{-\beta_t} N^{a_kt,\gamma_t}_\mathrm{good}
	\right) \right)}{\sF_t} \nonumber \\
& \xrightarrow[t\to\infty]{\P} 
\Ecsq{\exp \left( i \sum_{k=1}^n \lambda_k S_{Z_\infty/\sqrt{a_k}}\right)}
	{Z_\infty},\label{de}
\end{align}
where $(S_t)_{t\ge0}$ is the L\'evy process from the statement of Theorem~\ref{theorem}.
Recall that $\cL^{at,\gamma_t}_\mathrm{good} = \{ u \in \cL^{t,\gamma_t}_\mathrm{good} : \Delta_u^{t,\gamma_t} > at\}$ and that, for $u \in \cL^{at,\gamma_t}_\mathrm{good}$, we have $\Delta_u^{at,\gamma_t} = \Delta_u^{t,\gamma_t}$ and, therefore $Z_\infty^{(u,at,\gamma_t)} = Z_\infty^{(u,t,\gamma_t)}$.
Thus, we have 
\begin{align*}
F^{at,\gamma_t}_\mathrm{good} = 
\sum_{u \in \cL^{t,\gamma_t}_\mathrm{good} \text{ s.t.\@ } \Delta_u^{t,\gamma_t} > at} 
	\e^{-\gamma_t} Z_\infty^{(u,t,\gamma_t)},
\end{align*}
where, given $\sF_{\cL^{t,\gamma_t}}$, the $Z_\infty^{(u,t,\gamma_t)}$ for $u \in \cL^{t,\gamma_t}_\mathrm{good}$ are i.i.d.\@ copies of $Z_\infty$.
Recalling that $\Psi_{Z_\infty}$ denotes the characteristic function of $Z_\infty$, setting $\lambda_k' \coloneqq \lambda_1 + \dots + \lambda_k$ and $a_{n+1} = \infty$, it gives
\begin{align}
& \Ecsq{\exp \left( i \sum_{k=1}^n \lambda_k \left( 
	\sqrt{t} F^{a_k t,\gamma_t}_\mathrm{good} 
	- \beta_t \e^{-\beta_t} N^{a_k t,\gamma_t}_\mathrm{good}
	\right) \right)}{\sF_{\cL^{t,\gamma_t}}} \nonumber \\
& = \E \Biggl[ \exp \Biggl( i \sum_{k=1}^n \lambda_k'
	\sum_{u \in \cL^{t,\gamma_t}_\mathrm{good} 
		\text{ s.t.\@ } \Delta_u^{t,\gamma_t} \in (a_kt,a_{k+1}t]} 
	\e^{-\beta_t} \left( Z_\infty^{(u,t,\gamma_t)} - \beta_t \right) 
	\Biggr)
	\Bigg| \sF_{\cL^{t,\gamma_t}} \Biggr] \nonumber \\
& = \prod_{k=1}^n 
	\left[\Psi_{Z_\infty} \left( \lambda_k' \e^{-\beta_t} \right)\exp \left( -i \lambda_k' \beta_t \e^{-\beta_t}\right)\right]^{N^{a_kt,\gamma_t}_\mathrm{good} - N^{a_{k+1}t,\gamma_t}_\mathrm{good}}.
	\label{pa}
\end{align}
But, using \eqref{eq:characteristic-function-Z_infty} and Lemma~\ref{lem:Psi_algebra}, we have for every $\lambda \in \R$ and large enough $t$,
\begin{align*}
\Psi_{Z_\infty} \left( \lambda \e^{-\beta_t} \right)\exp \left(- i \lambda \beta_t \e^{-\beta_t}\right)
&= \Psi_{\pi/2,\mu_Z}(\lambda \e^{-\beta_t}) \exp \left(- i \lambda \beta_t \e^{-\beta_t} + \lambda \e^{-\beta_t} g(\lambda \e^{-\beta_t}) \right)\\
&= \exp \left( \e^{-\beta_t}\left[ -\psi_{\pi/2,\mu_Z}(\lambda) + \lambda g(\lambda \e^{-\beta_t})\right] \right),
\end{align*}
with $g(\lambda) \to 0$ as $\lambda \to 0$.
Therefore \eqref{pa} is equal to
\begin{align} \label{pb}
\exp \left( \sum_{k=1}^n \e^{-\beta_t}  \left(N^{a_kt,\gamma_t}_\mathrm{good} - N^{a_{k+1}t,\gamma_t}_\mathrm{good}\right) \left[  
	-\psi_{\pi/2,\mu_Z}(\lambda_k')
	+ \lambda_k' g \left( \lambda_k' \e^{-\beta_t} \right) 
	\right] \right).
\end{align}
Taking the conditional expectation given $\sF_t$ in \eqref{pb} and applying Corollary~\ref{cor:cv-N}, we get that the left-hand side of \eqref{de} converges in probability to
\begin{align*}
& \exp \left(- \sqrt{\frac{2}{\pi}} Z_\infty 
\sum_{k=1}^n \left( \frac{1}{\sqrt{a_k}} - \frac{1}{\sqrt{a_{k+1}}} \right) \psi_{\pi/2,\mu_Z}(\lambda_k') \right).
\end{align*}
But this is exactly the right-hand side of \eqref{de}, since $(S_t)_{t\ge0}$ is a L\'evy process and so has independent increments.
It concludes the proof.
\end{proof}

\section{Proof of Proposition~\ref{prop:control-W_t}}
\label{section:control-W}

In this section, we prove Proposition \ref{prop:control-W_t}, controlling the speed of convergence of $\sqrt{t} W_t$ towards $\sqrt{2/\pi} Z_\infty$. 
We use in the proof the following fact, which is a consequence of
Proposition~\ref{prop:one-dim} combined with \eqref{eq:convergence-W_t}:
\begin{align} \label{eq:first-control-Z_t-Z_infty}
\forall \varepsilon > 0, \quad \limsup_{t\to\infty}
\P \left( \abs{Z_t -Z_\infty} \geq t^{-1/2 + \varepsilon}\right) 
= 0.
\end{align}
The arguments used here are a very rough version of the arguments used in the forthcoming paper \cite{fluctuations2} for the fluctuations of the critical additive martingale.
\begin{proof}[Proof of Proposition \ref{prop:control-W_t}]
Fix some $\alpha \in (0,1)$; its value will later be chosen appropriately. Let
$t\ge2$. We introduce a killing barrier at 0 between times $t^\alpha$ and $t$:
we set, for $s \in [t^\alpha,t]$,
\begin{align*}
\overline{W}_s 
\coloneqq 
\sum_{u \in \cN(s)} \e^{-X_u(s)}
\1_{\forall r \in [t^\alpha,t], X_u(r) > 0}
\quad 
\text{and}
\quad
\overline{Z}_s
\coloneqq 
\sum_{u \in \cN(s)} X_u(s) \e^{-X_u(s)}
\1_{\forall r \in [t^\alpha,s], X_u(r) > 0}.
\end{align*}
The steps of the proof are the following. 
First note that, with high probability, none of the particles is killed, since
$\min_{u \in \cN(s)} X_u(s) \to \infty$ a.s.\@ as $s \to \infty$ by \eqref{eq:minimum}.
Therefore, we have 
\begin{align} \label{ef}
\P\left(W_t \neq \overline{W}_t\right) \xrightarrow[t\to\infty]{} 0
\end{align}
and we can consider $\overline{W}_t$ instead of $W_t$.
We will show that the conditional first moment $\E[\overline{W}_t | \sF_{t^\alpha}]$ is very close to $\sqrt{2/\pi} Z_{t^\alpha}$ and so to $\sqrt{2/\pi} Z_\infty$, by \eqref{eq:first-control-Z_t-Z_infty}.
Then, we prove that $\overline{W}_t$ is close to $\E[\overline{W}_t | \sF_{t^\alpha}]$ by a second moment argument, using results proved with the change of probability in Section~\ref{subsection:truncation-critical-additive-martingale}.
We will use parameters $\alpha$ and $\beta$, that will be fixed at the end, and
the constants $C$ can depend on them.

We first deal with the first moment. 
Recall from Lemma \ref{lem:control-first-moment-W} the definition of the function
$F(y) = \sqrt{2/\pi} \int_0^y \e^{-z^2/2} \diff z$, $y\ge0$. 
Applying the branching property at time
$t^\alpha$ and Lemma \ref{lem:control-first-moment-W}, we get
\begin{align*}
\Ecsq{\overline{W}_t}{\sF_{t^\alpha}}
= \sum_{v \in \cN(t^\alpha)}
	\e^{-X_v(t^\alpha)} 
	F \left( \frac{X_v(t^\alpha)}{\sqrt{t-t^\alpha}} \right) 
	\1_{X_v(t^\alpha) > 0}.
\end{align*}
Since $\lvert F(y) - \sqrt{2/\pi} y \rvert \leq C y^3$, for $y\ge0$, and $t-t^\alpha
\ge Ct$ for $t\ge2$, we get
\begin{align*}
\abs{ \Ecsq{\overline{W}_t}{\sF_{t^\alpha}} 
- \sqrt{\frac{2}{\pi (t- t^\alpha)}} \overline{Z}_{t^\alpha}}
& \leq 
C \sum_{v \in \cN(t^\alpha)}
\e^{-X_v(t^\alpha)} \left( \frac{X_v(t^\alpha)}{\sqrt{t-t^\alpha}} \right)^3
\1_{X_v(t^\alpha) > 0} \\
& \leq 
C t^{\alpha-(3/2)}
\sum_{v \in \cN(t^\alpha)}
\e^{-X_v(t^\alpha)} X_v(t^\alpha) 
\left( \frac{X_v(t^\alpha)}{\sqrt{t^\alpha}} \right)^2
\1_{X_v(t^\alpha) > 0}.
\end{align*}
Using furthermore that $|(t-t^\alpha)^{-1/2} - t^{-1/2}| \leq C t^{\alpha
-(3/2)}$, we have
\begin{align*}
\abs{\sqrt{t} \Ecsq{\overline{W}_t}{\sF_{t^\alpha}} 
- \sqrt{\frac{2}{\pi}} \overline{Z}_{t^\alpha}}
& \leq 
C t^{\alpha-1}
\sum_{v \in \cN(t^\alpha)}
\e^{-X_v(t^\alpha)} X_v(t^\alpha) 
\left(1 + \left( \frac{X_v(t^\alpha)}{\sqrt{t^\alpha}} \right)^2 \right)
\1_{X_v(t^\alpha) > 0}.
\end{align*}
We want to use Markov's inequality to bound the left-hand side of this
inequality in probability, but have to truncate the paths for this.
For any $\varepsilon >0$, there exists $K >0$ such that $\P (\min_{s \geq
0} \min_{u\in\cN(s)} X_u(s) < -K) \leq \varepsilon$, by \eqref{eq:minimum}. Decomposing with respect
to this event, then applying Markov's inequality, the many-to-one
formula and the relation \eqref{eq:link-between-R-and-B} between the Brownian
motion killed at 0 and the 3-dimensional Bessel process, we successively get,
\begin{align*}
& \Pp{\abs{\sqrt{t} \Ecsq{\overline{W}_t}{\sF_{t^\alpha}} 
- \sqrt{\frac{2}{\pi}} \overline{Z}_{t^\alpha}} \geq \frac A {t^{1-\alpha}}} \\
& \leq \varepsilon + \frac{C}{A} \Ec{\sum_{v \in \cN(t^\alpha)}
\e^{-X_v(t^\alpha)} X_v(t^\alpha) 
\left(1 + \left( \frac{X_v(t^\alpha)}{\sqrt{t^\alpha}} \right)^2 \right)
\1_{X_v(t^\alpha) > 0} 
\1_{\forall s \leq t^\alpha, X_u(s) \geq -K}} \\
& \leq \varepsilon + \frac{C}{A} \Eci{K}{B_{t^\alpha}
\left(1 + \left( \frac{B_{t^\alpha}}{\sqrt{t^\alpha}} \right)^2 \right)
\1_{\forall s \leq t^\alpha, B_s \geq 0}} \\
&= \varepsilon + \frac{CK}{A} \Eci{K/\sqrt{t^\alpha}}{1 + R_1^2}
\xrightarrow[t\to\infty]{}
\varepsilon + \frac{CK}{A} \Ec{1 + R_1^2}.
\end{align*}
Combining this with the fact that $\P(Z_{t^\alpha} \neq
\overline{Z}_{t^\alpha}) \to 0$ as $t \to \infty$, we finally get
\begin{align} \label{ea}
\limsup_{A \to \infty}
\limsup_{t \to \infty}
\Pp{\abs{\sqrt{t} \Ecsq{\overline{W}_t}{\sF_{t^\alpha}} 
- \sqrt{\frac{2}{\pi}} Z_{t^\alpha}} \geq \frac{A}{t^{1-\alpha}} }
= 0,
\end{align}
by taking $\varepsilon \to 0$.
We now want to prove that $\E [ \overline{W}_t | \sF_{t^\alpha}]$ is close to $\overline{W}_t$. 
By analogy with Section~\ref{subsection:truncation-critical-additive-martingale}, we introduce, for $\kappa \geq 1$, 
\begin{align*}
\overline{B}_{t,\kappa}
&\coloneqq
\left\{
u \in \cN(t) :
\forall v < u, d_v \geq t^\alpha \Rightarrow O_v \leq \kappa \e^{X_v(d_v)/2}
\right\} \\
\overline{W}_{t,\kappa}
& \coloneqq 
\sum_{u \in \cN(t)} \e^{-X_u(t)} 
\1_{\forall s \in [t^\alpha,t], X_u(s) > 0} \1_{u \in \overline{B}_{t,\kappa}}.
\end{align*}
Then, by the triangle inequality and noting that $\overline{W}_t -
\overline{W}_{t,\kappa} \geq 0$, we have, for any $\varepsilon > 0$,
\begin{align*}
& \Ppsq{\abs{\overline{W}_t - \Ecsq{\overline{W}_t}{\sF_{t^\alpha}}} \geq 3
\varepsilon}
	{\sF_{t^\alpha}}\\
& \leq 
\Ppsq{\overline{W}_t - \overline{W}_{t,\kappa} \geq \varepsilon}
	{\sF_{t^\alpha}}
+ \Ppsq{\abs{\overline{W}_{t,\kappa} -
\Ecsq{\overline{W}_{t,\kappa}}{\sF_{t^\alpha}}} \geq \varepsilon}
	{\sF_{t^\alpha}}
+ \1_{\Ecsq{\overline{W}_t - \overline{W}_{t,\kappa}}{\sF_{t^\alpha}} \geq
\varepsilon}.
\end{align*}
Applying Markov's inequality, Chebyshev's inequality and the inequality
$\1_{x\ge \varepsilon} \le x/\varepsilon$ for $x\ge0$, we get
\begin{align}
\Ppsq{\abs{\overline{W}_t - \Ecsq{\overline{W}_t}{\sF_{t^\alpha}}} \geq 3
\varepsilon}
	{\sF_{t^\alpha}} 
\leq 2 \cdot \frac{1}{\varepsilon} 
	\Ecsq{\overline{W}_t - \overline{W}_{t,\kappa}}{\sF_{t^\alpha}}
+ \frac{1}{\varepsilon^2} \Varsq{\overline{W}_{t,\kappa}}{\sF_{t^\alpha}}.
\label{ee}
\end{align}
In order to bound these two terms, we use the branching property at time
$t^\alpha$ and then Lemma \ref{lem:not-in-B-for-W} for the first term and
Lemma~\ref{lem:second-moment-for-W} for the second. This gives
\begin{align*}
\Ecsq{\overline{W}_t - \overline{W}_{t,\kappa}}{\sF_{t^\alpha}}
& \leq \sum_{v \in \cN(t^\alpha)}
	C h(\kappa) \e^{-X_v(t^\alpha)} 
	\left( \frac{1}{\sqrt{t-t^\alpha}} + \frac{X_v(t^\alpha)}{t-t^\alpha}
\right)
	\1_{X_v(t^\alpha) > 0} \\
& = C h(\kappa) 
	\left( \frac{\overline{W}_{t^\alpha}}{\sqrt{t-t^\alpha}} 
	+ \frac{\overline{Z}_{t^\alpha}}{t-t^\alpha} \right), \\
\Varsq{\overline{W}_{t,\kappa}}{\sF_{t^\alpha}}
&\leq \sum_{v \in \cN(t^\alpha)} 
	\Eci{X_v(t^\alpha)}{(\widetilde{W}_{t-t^\alpha,\kappa})^2}
\leq C \kappa 
	\left( \frac{\overline{W}_{t^\alpha}}{t-t^\alpha} 
	+ \frac{\overline{Z}_{t^\alpha} \ln (t-t^\alpha)}{(t-t^\alpha)^{3/2}} \right).
\end{align*}
We plug these bounds into \eqref{ee} with $\kappa \coloneqq 1$ and $\varepsilon
\coloneqq t^{-\beta - \frac{1}{2}}$ for some $\beta \in
(0,\alpha/4)$. Using again the inequality $t-t^\alpha \ge Ct$ for $t\ge2$, this
gives
\begin{align*}
\Ppsq{\sqrt{t} \abs{\overline{W}_t - \Ecsq{\overline{W}_t}{\sF_{t^\alpha}}} 
	\geq 3 t^{-\beta}}
	{\sF_{t^\alpha}}
& \leq C 
\left( \overline{W}_{t^\alpha} \left(
t^\beta + t^{2\beta}\right)
	+ \overline{Z}_{t^\alpha} 
		\left( t^{\beta-\frac 1 2} + t^{2\beta-\frac 1 2} \ln
t \right) \right)
	\nonumber \\
& \leq C 
t^{2\beta - \frac{\alpha}{2}}
\left( t^{\alpha/2} W_{t^\alpha} 
	+ Z_{t^\alpha}  \right)
\xrightarrow[t \to \infty]{\P} 0,
\end{align*}
using \eqref{eq:convergence-Z_t} and \eqref{eq:convergence-W_t}.
It follows that
\begin{align} \label{ed}
\Pp{\sqrt{t} \abs{\overline{W}_t - \Ecsq{\overline{W}_t}{\sF_{t^\alpha}}} 
	\geq 3 t^{-\beta}}
\xrightarrow[t \to \infty]{} 0.
\end{align}
Finally, we consider some $\theta \in (0, 1/5)$ and choose $\alpha \in (4 \theta, 1-\theta)$ and $\beta \in (\theta, \alpha/4)$.
Then, for $t$ large enough, we have $t^{-\theta} \geq 3 t^{-\beta} +
t^{1-\theta-\alpha}/(3 t^{1-\alpha}) + t^{\frac{\alpha}{2} - \theta}
/(3\sqrt{t^\alpha})$ and, therefore,
\begin{align*}
& \limsup_{t\to\infty}
\P \left( \abs{\sqrt{t} W_t - \sqrt{\frac{2}{\pi}} Z_\infty} 
	\geq t^{-\theta} \right) \\
\begin{split}
& \leq \limsup_{t\to\infty}
\P \left( W_t \neq \overline{W}_t \right)
+ \P \left( \sqrt{t} \abs{\overline{W}_t - 
	\Ecsq{\overline{W}_t}{\sF_{t^\alpha}}} 
	\geq 3 t^{-\beta} \right) \\
& \relphantom{\leq} \hphantom{\limsup_{t\to\infty}} {}
+ \P \left( \abs{\sqrt{t} \Ecsq{\overline{W}_t}{\sF_{t^\alpha}} 
	- \sqrt{\frac{2}{\pi}} Z_{t^\alpha}} 
	\geq \frac{t^{1-\theta-\alpha}}{3 t^{1-\alpha}} \right)
+ \P \left( \sqrt{\frac{2}{\pi}} \abs{Z_\infty - Z_{t^\alpha}} 
	\geq \frac{t^{\frac{\alpha}{2} - \theta}}{3\sqrt{t^\alpha}} \right)
\end{split} \\
& = 0,
\end{align*}
applying \eqref{ef}, \eqref{ed}, \eqref{ea} (because $t^{1-\theta-\alpha} \to
\infty$) and \eqref{eq:first-control-Z_t-Z_infty} (because $\alpha/2-\theta >
0$).
\end{proof}

\appendix

\section{Weak convergence in probability}
\label{section:weak-convergence-in-probability}

We work here on a Polish space $E$ with its Borel algebra $\cE$.
We denote by $\cC_b(E)$ the set of bounded continuous functions from $E \to \R$.
For a finite measure $\xi$ on $(E,\cE)$ and a function $f \in \cC_b(E)$, we set $\xi(f) \coloneqq \int_E f \diff \xi$.
Let $(\mu_n)_{n \in \N}$ be a sequence of random probability measures on $(E,\cE)$. 
We say that $\mu_n$ \textit{converges weakly almost surely} 
to a random probability measure $\mu_\infty$ as $n \to \infty$ if 
\begin{align*}
\text{a.s.},\quad \forall f \in \cC_b(E),\quad
\mu_n(f) \xrightarrow[n\to\infty]{} \mu_\infty(f).
\end{align*}
Berti, Pratelli and Rigo \cite{bpr2006} proved that the two following statements are equivalent:
\begin{enumerate}
\item[(i)] for any $f \in \cC_b(E)$, $\mu_n(f)$ converges a.s. as $n \to \infty$;
\item[(ii)] there exists a random probability measure $\mu_\infty$ such that $\mu_n \to \mu_\infty$ weakly almost surely.
\end{enumerate}
The point here is the interchange of ``a.s.'' and the quantifier ``$\forall f$''.

We say that $\mu_n$ \textit{converges weakly in probability} 
to $\mu_\infty$ as $n \to \infty$ if, for any subsequence $(n_k)_{k\in\N}$, there exists a subsequence $(n_{k_i})_{i\in\N}$ such that $\mu_{n_{k_i}}$ converges to $\mu_\infty$ weakly almost surely as $i\to\infty$.
Then, Berti, Pratelli and Rigo \cite{bpr2006} also showed the equivalence between the following statements:
\begin{enumerate}
\item[(i)'] for any $f \in \cC_b(E)$, $\mu_n(f)$ converges in probability as $n \to \infty$;
\item[(ii)'] there exists a random probability measure $\mu_\infty$ such that $\mu_n \to \mu_\infty$ weakly in probability.
\end{enumerate}
For $n \in [0,\infty)$, we will denote by $\P\mu_n$ 
the \textit{annealed} probability measure defined by 
\[
\forall A \in \cE, \quad \P\mu_n(A) \coloneqq \Ec{\mu_n(A)}.
\]
Note that, if $\mu_n$ converges weakly in probability to $\mu_\infty$,
then $\P\mu_n$ converges weakly to $\P\mu_\infty$.

We now work on the space $E = \R^d$ and establish the following proposition.
An analog result for the weak convergence almost surely has been proved by Berti, Pratelli and Rigo \cite{bpr2006}.
\begin{prop} \label{prop:weak-convergence-in-probability}
Let $\mu_n$ for $n \in \N \cup \{ \infty \}$ be random probability measures on $\R^d$. 
Then, $\mu_n$ converges weakly in probability to $\mu_\infty$ iff, 
for any $\lambda \in \R^d$, we have 
\begin{align*}
\int_{\R^d} \e^{i (\lambda,x)} \diff \mu_n(x) 
\xrightarrow[n\to\infty]{}
\int_{\R^d} \e^{i (\lambda,x)} \diff \mu_\infty(x),
\quad \text{in probability}.
\end{align*}
\end{prop}
\begin{proof}
For $\lambda \in \R^d$, we set $f_\lambda \colon x \in \R^d \mapsto \e^{i (\lambda,x)}$.
The direct implication is obvious, so we prove the reciprocal: 
we assume that, for any $\lambda \in \R^d$, 
$\mu_n(f_\lambda) \to \mu_\infty(f_\lambda)$ in probability.
Let $(n_k)_{k\in\N}$ be a subsequence, we want to prove that there exists a subsequence $(n_{k_i})_{i\in\N}$ of $(n_k)_{k\in\N}$ such that a.s., $\forall \lambda \in \R^d$, 
$\mu_{n_{k_i}} (f_\lambda) \to \mu_\infty(f_\lambda)$ as $i \to \infty$.
For $n \in \N \cup\{\infty\}$ and $\delta > 0$, we set
\begin{align*}
Y_n (\delta) 
\coloneqq 
\int_{\R^d} \left( 2 \wedge \delta \abs{x} \right) \diff \mu_n (x),
\end{align*}
so that, if $\abs{\lambda - \lambda'} \leq \delta$, then $\abs{\mu_n(f_\lambda)-\mu_\infty(f_\lambda)} 
\leq \abs{\mu_n(f_{\lambda'})-\mu_\infty(f_{\lambda'})}
+ Y_n (\delta) + Y_\infty (\delta)$.
The assumption implies that, for any $\lambda \in \R^d$, 
$\P\mu_n(f_\lambda) \to \P\mu_\infty(f_\lambda)$
and therefore that $\P\mu_n$ converges weakly to $\P\mu_\infty$.
Thus, it follows that $\Ec{Y_n (\delta)} \to \Ec{Y_\infty (\delta)}$ as $n \to \infty$. Moreover, by dominated convergence, for any $n \in \N \cup\{\infty\}$, we have $\Ec{Y_n (\delta)} \downarrow 0$ as $\delta \downarrow 0$. 
Combining this and the fact that $Y_n(\delta)$ is increasing in $\delta$ for every $n$, one can easily get that
\begin{align*}
\sup_{n\in\N \cup \{\infty\}} \Ec{Y_n (\delta)} \xrightarrow[\delta \downarrow 0]{} 0.
\end{align*}
Therefore, for any $i \in \N^*$, there exists $\delta_i > 0$ such that 
$\sup_{n\in\N \cup \{\infty\}} \Ec{Y_n (\delta_i)} \leq 2^{-i}$.
Now, we set $\Lambda_i \coloneqq [-i,i]^d \cap \delta_i \Z^d$. 
We construct the subsequence $(n_{k_i})_{i\in\N}$ of $(n_k)_{k\in\N}$, by recurrence: let $n_{k_0} \coloneqq n_0$ and, for $i \geq 1$, we choose $n_{k_i} > n_{k_{i-1}}$ such that
\begin{align*}
\Pp{ \exists \lambda \in \Lambda_i : 
	\abs{\mu_{n_{k_i}}(f_\lambda)-\mu_\infty(f_\lambda)}
	\geq \frac{1}{3i}} 
\leq \frac{1}{2^i},
\end{align*}
using here the assumption.
Then, we get, for any $i \geq 1$,
\begin{align}
& \Pp{ \exists \lambda \in [-i,i]^d : 
	\abs{\mu_{n_{k_i}}(f_\lambda)-\mu_\infty(f_\lambda)}
	\geq \frac{1}{i}} \nonumber \\
& \leq 
\Pp{ \exists \lambda \in \Lambda_i : 
	\abs{\mu_{n_{k_i}}(f_\lambda)-\mu_\infty(f_\lambda)}
	\geq \frac{1}{3i}} 
+ \Pp{Y_{n_{k_i}}(\delta_i) \geq \frac{1}{3i}}
+ \Pp{Y_\infty(\delta_i) \geq \frac{1}{3i}} \nonumber \\
& \leq \frac{1}{2^i} + \frac{3i}{2^i} + \frac{3i}{2^i}, \label{dl}
\end{align}
using Markov's inequality and the definitions of $n_{k_i}$ and of $\delta_i$.
The right-hand side of \eqref{dl} is summable in $i \in \N^*$, so it follows by Borel-Cantelli lemma that a.s.\@ there exists $i_0 \geq 1$ such that, for any $i \geq i_0$ and $\lambda \in [-i,i]^d$, 
$\lvert \mu_{n_{k_i}}(f_\lambda)-\mu_\infty(f_\lambda) \rvert < 1/i$.
This implies that a.s.\@ for any $\lambda \in \R^d$, 
$\mu_{n_{k_i}}(f_\lambda) \to \mu_\infty(f_\lambda)$ as $i \to \infty$.
\end{proof}

As a corollary of this result, we state the following generalization of Slutsky's theorem.
\begin{cor}
Let $\mu_n$ for $n \in \N$ be random probability measures on $\R^{d+d'}$. 
We denote by $\mu_n^1$ the marginal distribution of $\mu_n$ associated with the $d$ first coordinates and $\mu_n^2$ associated with the $d'$ last coordinates.
Assume that $\mu_n^1$ converges weakly in probability to a random probability measure $\mu_\infty^1$ on $\R^d$ and that $\mu_n^2$ converges weakly in probability to $\delta_x$, for some $x \in \R^{d'}$.
Then, $\mu_n$ converges weakly in probability to $\mu_\infty^1 \otimes \delta_x$.
\end{cor}
\begin{rem} \label{rem:weak-convergence-in-proba}
One consequence of this corollary that we use repetitively in the paper is the following.
Let $X_n$ and $Y_n$ for $n \in \N$ be random variables taking values in $\R^d$ and $\sF_n$ for $n\in\N$ be $\sigma$-fields. If the conditional law of $X_n$ given $\sF_n$ converges weakly in probability to some random probability $\mu_\infty$ and $Y_n$ converges in probability to 0, then the conditional law of $X_n + Y_n$ given $\sF_n$ converges weakly in probability to $\mu_\infty$.
\end{rem}

\section{Some formulae for the three-dimensional Bessel process}
\label{section:technical-results}


Let $(B_s)_{s\geq 0}$ denote a standard Brownian motion starting from $x$ under $\P_x$ and $\tau \coloneqq \inf \{ r \geq 0 : B_r = 0 \}$.
Using that the Green function of the Brownian motion killed at 0 is $G(x,y) = 2 (x \wedge y) \leq 2y$ for $x,y >0$, one can get the two following bounds, for any $x,a >0$,
\begin{align}
\Eci{x}{\int_0^\tau \1_{B_r < a} \diff r} & \leq a^2, \label{eq:time-killed-BM-under-a} 
\\
\Eci{x}{\int_0^\tau \e^{-B_t/2} \diff t}& \leq 8, \label{eq:integral-killed-BM}
\end{align}
obtained by a direct computation.


Let $(R_s)_{s\geq 0}$ denote a 3-dimensional Bessel process starting from $x$ under $\P_x$. 
Recall that, for $x > 0$, one has the following link between the 3-dimensional Bessel process and the Brownian motion (see Imhof \cite{imhof84}): for any $t \geq 0$ and any measurable function $F \colon \cC([0,t]) \to \R_+$,
\begin{equation} \label{eq:link-between-R-and-B}
\Eci{x}{F(B_s,s\in [0,t]) \1_{\forall s \in [0,t], B_s > 0}}
=
\Eci{x}{\frac{x}{R_s} F(R_s,s\in [0,t])}.
\end{equation}
The density of $R_t$ under $\P_x$ is 
\begin{equation}
 \label{eq:density_R}
 z\mapsto \frac{\e^{-(z-x)^2/2t}}{\sqrt{2\pi}} \1_{z >0}\times
	  \begin{cases}
	    \frac{z}{x\sqrt{t}} (1-\e^{-2 xz/t}) & \text{if $x>0$}\\
	    \frac{2 z^2}{t^{3/2}} & \text{if $x=0$}.
          \end{cases}
\end{equation}
The Green function of the 3-dimensional Bessel process is $G(x,y) = 2 y^2
(x^{-1} \wedge y^{-1})$ for $x,y >0$, and the two following bounds follow: for
any $a,x>0$
\begin{align}
\Eci{x}{\int_0^\infty \1_{R_r < a} \diff r} \leq \frac{a^3}{x}, 
	\label{eq:time-bessel-under-a} 
\\
\Eci{x}{\int_0^\infty R_r \e^{-R_r/2} \diff r} \leq \frac{32}{x}. 
	\label{eq:integral-bessel}
\end{align}
We now  establish three lemmas that are slightly more technical.
\begin{lem} \label{lem:time-bessel-under-a-and-inverse}
There exists $C>0$ such that, for any $s,x,a >0$, we have
\[
\Eci{x}{\frac{1}{R_s} \int_0^s \1_{R_r < a} \diff r}
\leq C a^3 \left( \frac{1}{x \sqrt{s}} + \frac{1}{s} \right).
\]
\end{lem}
\begin{proof}
First note that, for any $y,t \geq 0$, $\Eci{y}{1/R_t} \leq \Ec{1/R_t} = \sqrt{2/\pi t}$.
We cut the integral into two pieces and first deal with the part $r \in [0,s/2]$: we have, by Markov's property at time $s/2$,
\begin{align}
\Eci{x}{\frac{1}{R_s} \int_0^\frac{s}{2} \1_{R_r < a} \diff r}
& = \Eci{x}{\Eci{R_{s/2}}{\frac{1}{R_{s/2}}} \int_0^\frac{s}{2} \1_{R_r < a} \diff r} 
	\nonumber \\
& \leq \sqrt{\frac{4}{\pi s}} \Eci{x}{\int_0^\frac{s}{2} \1_{R_r < a} \diff r} 
\leq \frac{2}{\sqrt{\pi s}} \frac{a^3}{x}, \label{cd}
\end{align}
by applying \eqref{eq:time-bessel-under-a}.
Now, we deal with the second part of the integral:
\begin{align} \label{ce}
\Eci{x}{\frac{1}{R_s} \int_\frac{s}{2}^s \1_{R_r < a} \diff r}
& = \int_\frac{s}{2}^s \Eci{x}{\frac{1}{R_s} \1_{R_r < a}} \diff r
= \int_\frac{s}{2}^s \Eci{x}{\Eci{R_r}{\frac{1}{R_{s-r}}} \1_{R_r < a}} \diff r \nonumber \\
& \leq \int_\frac{s}{2}^s \sqrt{\frac{2}{\pi (s-r)}} \Ppi{x}{R_r < a} \diff r.
\end{align}
Moreover, we have
\begin{align} \label{cf}
\Ppi{x}{R_r < a} 
\leq \Pp{R_r < a} 
= \int_0^a \sqrt{\frac{2}{\pi}} \frac{z^2}{r^{3/2}} \e^{-z^2/2r} \diff z
\leq \sqrt{\frac{2}{\pi}} \left( \frac{2}{s} \right)^{3/2} \int_0^a z^2 \diff z
= \frac{4 a^3}{3 \sqrt{\pi} s^{3/2}},
\end{align}
by using that $r \geq s/2$.
Thus, we get
\begin{align*}
\Eci{x}{\frac{1}{R_s} \int_\frac{s}{2}^s \1_{R_r < a} \diff r}
& \leq \frac{4 \sqrt{2} a^3}{3 \pi s^{3/2}} \int_\frac{s}{2}^s \frac{\diff r}{\sqrt{s-r}}
= \frac{8 a^3}{3 \pi s},
\end{align*}
and it concludes the proof.
\end{proof}
\begin{lem} \label{lem:integral-bessel-and-inverse}
There exists $C>0$ such that, for any $x,a >0$ and $s \geq 2$, we have
\[
\Eci{x}{\frac{1}{R_s} \int_0^s \left(\frac{R_r}{\sqrt{s-r}} \wedge 1 \right) 
	\e^{-R_r/2} \diff r}
\leq C \left( \frac{1}{xs} + \frac{\ln s}{s^{3/2}} \right).
\]
\end{lem}
\begin{proof}
We have
\begin{align*}
\Eci{x}{\frac{1}{R_s} \int_0^s \left(\frac{R_r}{\sqrt{s-r}} \wedge 1 \right) 
	\e^{-R_r/2} \diff r}
& \leq \Eci{x}{\frac{1}{R_s} \int_0^s 
	\sum_{k \geq 0} \left(\frac{k+1}{\sqrt{s-r}} \wedge 1 \right) 
	\e^{-k/2} \1_{R_r \in [k,k+1)} \diff r} \\
& \leq \sum_{k \geq 0} (k+1) \e^{-k/2}
\Eci{x}{\frac{1}{R_s} \int_0^s \left(\frac{1}{\sqrt{s-r}} \wedge 1 \right) 
	\1_{R_r < k+1} \diff r}.
\end{align*}
As in the proof of Lemma \ref{lem:time-bessel-under-a-and-inverse}, we split the integral into two pieces: we have
\begin{align*}
& \Eci{x}{\frac{1}{R_s} \int_0^s \left(\frac{1}{\sqrt{s-r}} \wedge 1 \right) 
	\1_{R_r < k+1} \diff r} \\
& \leq \Eci{x}{\frac{1}{R_s} \int_0^\frac{s}{2} \sqrt{\frac{2}{s}} \1_{R_r < k+1} \diff r}
+ \int_\frac{s}{2}^s \left(\frac{1}{\sqrt{s-r}} \wedge 1 \right) 
	\Eci{x}{\frac{1}{R_s} \1_{R_r < k+1}} \diff r \\
& \leq \sqrt{\frac{2}{s}} \frac{2}{\sqrt{\pi s}} \frac{(k+1)^3}{x}
+ \int_\frac{s}{2}^s \left(\frac{1}{\sqrt{s-r}} \wedge 1 \right)  
	\frac{4 \sqrt{2} (k+1)^3}{3 \pi s^{3/2}\sqrt{s-r}} \diff r,
\end{align*}
using \eqref{cd} for the first term and \eqref{ce} with \eqref{cf} for the second term.
Noting that, for $s \geq 2$,
\begin{align*}
\int_\frac{s}{2}^s \left(\frac{1}{\sqrt{s-r}} \wedge 1 \right) \frac{1}{\sqrt{s-r}} \diff r
= \int_0^1 \frac{\diff r}{\sqrt{r}} + \int_1^\frac{s}{2} \frac{\diff r}{r}
\leq C \ln s,
\end{align*}
we finally get
\begin{align*}
\Eci{x}{\frac{1}{R_s} \int_0^s \left(\frac{R_r}{\sqrt{s-r}} \wedge 1 \right) 
	R_r \e^{-R_r/2} \diff r} 
& \leq C \sum_{k \geq 0} (k+1)^4 \e^{-k/2}
	\left( \frac{1}{x s} + \frac{\ln s}{s^{3/2}} \right),
\end{align*}
and it concludes the proof.
\end{proof}
\begin{lem} \label{lem:min-bessel-and-inverse}
There exists $C>0$ such that, for any $x,y >0$, $t > 0$ and $s\ge0$, we have
\begin{align*}
\Eci{x}{\frac{1}{R_{t+s}} \1_{\min_{r \in [t,t+s]} R_r \leq y}}
& \leq C\left(\frac{y^2}{t^{3/2}} 
+ \frac 1 {\sqrt t}\left(\frac{y}{\sqrt{s}} \wedge 1 \right)\1_{s>0}\right).
\end{align*}
\end{lem}

\begin{proof}
First note that the law of the process $(R_t)_{t\ge0}$ is stochastically increasing in $x$, as can be seen by a coupling argument. Moreover, the term in the expectation is a decreasing function of the path $(R_t)_{t\ge0}$. We can therefore bound the expectation by its value for $x=0$, which we assume from now on.

Using Markov's property at time $t$, we get
\begin{align}
\Ec{\frac{1}{R_{t+s}} \1_{\min_{r \in [t,t+s]} R_r \leq y}} 
&= \Ec{\Eci{R_t}{\frac{1}{R_s} \1_{\min_{r \in [0,s]} R_r \leq y}}}\nonumber\\
&= \Ec{\Eci{R_t}{\frac{1}{R_s}}\1_{R_t \leq y}} + \Ec{\Eci{R_t}{\frac{1}{R_s} \1_{\min_{r \in [0,s]} R_r \leq y}}\1_{R_t > y}}\nonumber\\
&\eqqcolon T_1+T_2.
	\label{yd}
\end{align}
Since $x\mapsto 1/x$ is a superharmonic function for the Bessel process, the first term is bounded by
\begin{align}
T_1 &\le \Ec{\frac 1 {R_t} \1_{R_t \leq y}} = \frac 1 {\sqrt t} \Ec{\frac 1 {R_1} \1_{R_1 \leq y/\sqrt t}} = \frac 1 {\sqrt t} \int_0^{y/\sqrt t} \frac{1}{z} \sqrt{\frac{2}{\pi}} z^2 \e^{-z^2/2} \diff z 
\leq C \frac{y^2}{t^{3/2}}.
 \label{yf}
\end{align}
As for the term $T_2$ in \eqref{yd}, first note that it vanishes if $s=0$. If $s>0$, we use that for $z \geq y > 0$, by  Equation 5.1.2.8 of Borodin and Salminen \cite{borodinsalminen2002},
\begin{align*}
\Eci{z}{\frac{1}{R_s} \1_{\min_{r \in [0,s]} R_r \leq y}}
& = \frac{1}{\sqrt{2\pi s} z} 
\int_0^\infty \left( \e^{-(u+z-2y)^2/2s} - \e^{-(u+z)^2/2s} \right) \diff u \\
& = \frac{1}{\sqrt{2\pi s} z} 
\int_{z-y}^\infty \left( \e^{-(u-y)^2/2s} - \e^{-(u+y)^2/2s} \right)  \diff u \\
& \leq \frac 1 z \times \frac{1}{\sqrt{2\pi s}} 
\int_0^\infty   \left( \e^{-(u-y)^2/2s} - \e^{-(u+y)^2/2s} \right) \diff u,
\end{align*}
using that $z\ge y$ in the last inequality. The second term in the product on the right-hand side equals the probability that a Brownian motion starting from $y$ does not hit 0 by time $s$ (see e.g. Appendix 1.3 in Borodin and Salminen \cite{borodinsalminen2002}) and is easily bounded by $C(y/\sqrt s \wedge 1)$.
Therefore, if $s>0$,
\begin{align*}
T_2 
\leq C \Ec{\frac 1 {R_t}} \left( \frac{y}{\sqrt{s}} \wedge 1 \right) 
\leq \frac{C}{\sqrt{t}} \left( \frac{y}{\sqrt{s}} \wedge 1  \right) ,
\end{align*}
by the scaling invariance of the Bessel process. Together with \eqref{yd} and \eqref{yf}, this proves the lemma.
\end{proof}
\section{Asymptotic of \texorpdfstring{$\Psi_{Z_\infty}$}{Psi\_\{Z\_infty\}}}
\label{section:asymptotic_Psi}

\def\ep{\varepsilon}

The following lemma can be deduced from classical results about the domain of attraction of (1-)stable laws, but it is difficult to find good references where all constants are explicit. We therefore prove it here for convenience.
\begin{lem}
\label{lem:asymptotic}
 Let $Z\ge0$ be a random variable satisfying
 \begin{enumerate}
  \item $\P(Z>x)\sim 1/x$ as $x\to\infty$,
  \item $\int_0^x \P(Z > y)\diff y - \log x \to c$ as $x\to\infty$, for some $c\in\R$.
 \end{enumerate}
Let $\gamma$ denote the Euler--Mascheroni constant. Then, as $\lambda \to 0$ in $\R$,
\begin{equation}
\label{eq:eilZ}
 \Ec{\e^{i\lambda Z}} = \exp\left(-\frac \pi 2 |\lambda| + i \lambda (-\log |\lambda| + c - \gamma) + o(|\lambda|)\right),
\end{equation}
where $o(|\lambda|)$ is a (complex-valued) term that satisfies $|o(|\lambda|)|/|\lambda| \to0$ as $\lambda \to 0$.
\end{lem}
\begin{proof}
Since $\E[\e^{-i\lambda Z}] = \overline{\E[\e^{i\lambda Z}]}$, it suffices to consider $\lambda > 0$. We will separate the regions where $Z \ll \lambda^{-1}$, $Z \asymp \lambda^{-1}$ and $Z \gg \lambda^{-1}$. For this, fix $\ep > 0$. Throughout the proof, we use the Landau symbols $o$ and $O$ which have their usual meaning, with the slight twist that the $o$ symbol may depend on $\ep$.
We first recall the following formula easily obtained by integration by parts (formally, using Fubini's theorem):
\[
\forall y\ge0, \quad \Ec{\e^{i\lambda (Z\wedge y)}} = 1 + i\lambda \int_0^y \e^{i\lambda x} \P(Z > x)\diff x.
\]
This allows to split up $\E[\e^{i\lambda Z}]$ as follows: using hypothesis 1,
\begin{align}
\Ec{\e^{i\lambda Z}} 
&= \Ec{\e^{i\lambda (Z\wedge (\ep\lambda)^{-1})}}
	+ O \left(\P \left(Z > (\ep\lambda)^{-1} \right) \right)
= 1 + i\lambda(I_1+I_2) + O(\ep\lambda), \label{eq:eilZ2}
\end{align}
where we set
\begin{align*}
I_1 \coloneqq \int_0^{(\ep\lambda)^{-1}} \e^{i\lambda x}\P(Z>x)\diff x
\quad \text{and} \quad
I_2 \coloneqq \int_{\ep\lambda^{-1}}^{(\ep\lambda)^{-1}} \e^{i\lambda x}\P(Z>x)\diff x.
\end{align*}
It remains to estimate the integrals $I_1$ and $I_2$. For the first integral, we have by hypotheses 1 and 2, as $\lambda \to 0$,
\begin{equation}
I_1 = \int_0^{\ep\lambda^{-1}} (1+O(\lambda x)) \P(Z>x)\diff x
= \log \left( \ep\lambda^{-1} \right) + c + o(1) + O(\ep).
\label{eq:I1}
\end{equation}
For the second integral, we have by a change of variables,
\[
I_2 = \int_\ep^{\ep^{-1}} \e^{ix} \Pp{Z > \lambda^{-1} x} \frac{\diff x}{\lambda}.
\]
Together with hypothesis 1 and dominated convergence, this gives as $\lambda \to 0$,
\begin{equation}
I_2 = \int_\ep^{\ep^{-1}} \e^{ix} \frac{\diff x}{x} + o(1).
\label{eq:I2}
\end{equation}
Recall (\cite[pp.~228ff]{Abramowitz1964}) that the \emph{exponential integral} $E_1$ is defined for $\lvert \arg(z) \rvert < \pi$ by $E_1(z) = \int_z^\infty \e^{-t}\frac{\diff t}{t}$, and that it satisfies
\[
E_1(z) = -\gamma - \log z + O(z)\quad (|z|\to0),\qquad E_1(z) \sim \frac{\e^{-z}}z\quad (|z|\to\infty).
\]
In particular, this gives
\[
\int_\ep^{\ep^{-1}} \e^{ix}\frac{\diff x}{x} = E_1(- i\ep) - E_1(- i\ep^{-1}) = -\gamma - \log \ep + \frac \pi 2 i + O(\ep),
\]
and plugging this into \eqref{eq:I2} gives
\begin{equation}
I_2 = -\gamma - \log \ep + \frac \pi 2 i + O(\ep) + o(1).
\label{eq:I2bis}
\end{equation}
Equations \eqref{eq:eilZ2}, \eqref{eq:I1} and \eqref{eq:I2bis} now give as $\lambda \to 0$,
\[
\Ec{\e^{i\lambda Z}} = 1 + i\lambda\left(- \log(\lambda) + c -\gamma +  \frac \pi 2 i + O(\ep) + o(1)\right).
\]
Letting first $\lambda\to0$ then $\ep\to0$ gives the desired result.
\end{proof}

\section{Rate of convergence of the derivative martingale}
\label{section:rate-of-cv-derivative-martingale}

In this last section, we prove Proposition \ref{prop:speed-of-CV-Z_t}.
For this, we need to recall two explicit bounds.
The first one concerns the global minimum of the BBM: for any $L > 0$,
\begin{align} \label{eq:global-min-of-the-BBM}
\Pp{\exists s \geq 0 : \min_{u \in \cN(s)} X_u(s) \leq - L}
\leq \e^{-L}.
\end{align}
This follows easily from Doob's inequality applied to the martingale $(W_t)_{t\ge0}$. We remark that more precise estimates have been proved recently under additional assumptions by Madaule \cite{madaule2016arxiv} (for the branching random walk) and Berestycki et al.\@ \cite{bbhm2016arxiv} (for the binary branching Brownian motion).

The second bound deals with the minimum of the BBM at time $t$: for any $x \in [0,\sqrt{t}]$ and $t \geq 2$, we have
\begin{align} \label{eq:local-min-of-the-BBM}
\Pp{\min_{u \in \cN(t)} X_u(t) \leq \frac{3}{2} \log t - x}
\leq C (1+x)^2 \e^{-x}.
\end{align}
This has been proved by Bramson \cite[Proposition 3]{bramson78} in the case $\E[L^2]<\infty$, but since the proof relies only on first moments arguments, it holds also under \eqref{hypothese 1}.

\begin{proof}[Proof of Proposition \ref{prop:speed-of-CV-Z_t}]
First note that we can assume that $\delta \geq t^{-1/2}$, because otherwise it is enough to bound the probability by 1.
We keep the notation of Section \ref{section:proof-of-theorem}, but take here $\beta_t = \frac 1 2 \log t$, so that $\gamma_t = \log t$.
Using \eqref{eq:decompo-Z_infty-2}, it is sufficient to prove the following inequalities
\begin{align}
\Pp{\abs{Z_t - \widetilde{Z}_t^{t,\gamma_t}} \geq \delta}
& \leq C \frac{(\log t)^2}{\delta \sqrt{t}}, \label{xz} \\
\Pp{\abs{Z_t^{t,\gamma_t} - \widetilde{Z}_\infty^{t,\gamma_t}} \geq \delta}
& \leq C \frac{\log t}{\delta \sqrt{t}}, \label{xa} \\
\Pp{F_\mathrm{good}^{t,\gamma_t} \geq \delta}
& \leq C \frac{(\log t)^2}{\delta \sqrt{t}}, \label{xb} \\
\Pp{\cL^{t,\gamma_t}_\text{bad} \neq \varnothing} 
& \leq C \frac{(\log t)^2}{\sqrt{t}}, \label{xc}
\end{align}
noting that \eqref{xc} is sufficient because $\delta \leq 1$.
We start with \eqref{xc}.
Using \eqref{eq:local-min-of-the-BBM}, we get
\begin{align*}
\Pp{\cL^{t,\gamma_t}_\text{bad} \neq \varnothing}
& = \Pp{\min_{u \in \cN(t)} X_u(t) \leq \log t}
\leq C \left( 1 + \frac{1}{2}\log t \right)^2 \frac{1}{\sqrt{t}},
\end{align*}
and it proves \eqref{xc}.

Now, we prove \eqref{xz}. 
On the event $\{\cL^{t,\gamma_t}_\text{bad} = \varnothing\}$, we have $\widetilde{Z}_t^{t,\gamma_t} = Z_t - (\log t) W_t$.
Therefore, using \eqref{eq:global-min-of-the-BBM} with $L= \log t > 0$ and using \eqref{xc}, we get
\begin{align*}
\Pp{\abs{Z_t - \widetilde{Z}_t^{t,\gamma_t}} \geq \delta}
& \leq C \frac{\log t}{\sqrt{t}} + \frac{1}{t}
	+ \Pp{(\log t) W_t \geq \delta,
		\min_{s \geq 0} \min_{u \in \cN(s)} X_u(s) \geq - \log t}.
\end{align*}
But, using the many-to-one formula, we have, for any $L > 0$,
\begin{align} \label{eq:control-first-moment-W_t-barrier-at--L}
\Ec{W_t \1_{	\min_{s \geq 0} \min_{u \in \cN(s)} X_u(s) \geq - L}}
\leq \Pp{\min_{s \in [0,t]} B_s \geq - L}
= \Pp{\abs{B_t} \leq L}
\leq C \frac{L}{\sqrt{t}}
\end{align}
and it shows \eqref{xz} with Markov's inequality.

We now deal with \eqref{xa}. 
For this, applying \eqref{ad} with $\kappa = 1$ and $\varepsilon = \delta\sqrt{t}$, we get 
\begin{align*}
\Ppsq{\abs{\widetilde{Z}_\infty^{t,\gamma_t} - \widetilde{Z}_t^{t,\gamma_t}} \geq \delta}{\sF_t}
& \leq C \sqrt{t} W_t \left( \frac{h(1)}{\delta\sqrt{t}} + \frac{\e^{-\beta_t}}{\delta^2 t} \right)	
\leq \frac{C}{\delta} W_t,
\end{align*}
using that $\delta \geq t^{-1/2}$ and $t \geq 2$.
Applying again \eqref{eq:global-min-of-the-BBM} with $L = \log t$ and \eqref{eq:control-first-moment-W_t-barrier-at--L}, we get
\begin{align*}
\Pp{\abs{\widetilde{Z}_s^{t,\gamma_t} - \widetilde{Z}_t^{t,\gamma_t}} \geq \delta}
& \leq \frac{1}{t}
+ \frac{C}{\delta}
	\Ec{W_t	\1_{\min_{s \geq 0} \min_{u \in \cN(s)} X_u(s) \geq - \log t}} 
\leq \frac{1}{t}
+ C \frac{\log t}{\delta \sqrt{t}}.
\end{align*}
It proves \eqref{xa}.

Finally, we deal with \eqref{xb}. 
By Markov's inequality,
\begin{align*}
\Ppsq{F_\mathrm{good}^{t,\gamma_t} \geq \delta}{\sF_{\cL^{t,\gamma_t}}}
& \leq \frac 1 \delta \Ecsq{\Biggl(\sum_{u \in \cL^{t,\gamma_t}_\mathrm{good}} 
	\frac{1}{t} 
	Z_\infty^{(u,t,\gamma_t)}\Biggr) \wedge 1}{\sF_{\cL^{t,\gamma_t}}}\\
& \leq \frac{\# \cL^{t,\gamma_t}_\mathrm{good}}{\delta t} 
	\Ec{Z_\infty \wedge t}\\
& \leq C N^{t,\gamma_t}_\mathrm{good} \frac{\log t}{\delta t},
\end{align*}
using \eqref{eq:tail-Z_infty-3}.
Applying again \eqref{eq:global-min-of-the-BBM} with $L = \log t$, we get
\begin{align*}
\Pp{F_\mathrm{good}^{t,\gamma_t} \geq \delta}
& \leq \frac{1}{t}
+ C \frac{\log t}{\delta t} 
	\Ec{N^{t,\gamma_t}_\mathrm{good} 
		\1_{\min_{s \in [0,t]} \min_{u \in \cN(s)} X_u(s) \geq - \log t}}.
\end{align*}
But, using \eqref{eq:first-moment-N_(1,infty)} with here $\beta_t = (\log t)/2$, we have 
$\E [N^{t,\gamma_t}_\mathrm{good} |\sF_t] = t \widetilde{W}^{t,\gamma_t}_t \leq t W_t$.
Therefore, using \eqref{eq:control-first-moment-W_t-barrier-at--L}, it follows that
\begin{align*}
\Pp{F_\mathrm{good}^{t,\gamma_t} \geq \delta}
& \leq \frac{1}{t}
+ C \frac{\log t}{\delta t} t \frac{\log t}{\sqrt{t}}
\end{align*}
and it proves \eqref{xb}.
\end{proof}

\label{FIN}
\addcontentsline{toc}{section}{References}
\bibliographystyle{abbrv}
\bibliography{biblio,fluctuations_mendeley}

\end{document}